\documentclass[11pt]{article}
\usepackage{etex}

\include{psfig}
\usepackage{pictex, latexsym, graphicx,amsmath,amssymb,amsbsy,amsfonts,amsthm,verbatim}
\usepackage[pdfpagemode=UseOutlines,colorlinks=true,pdfnewwindow=true,citecolor=blue,urlcolor=blue,linkcolor=blue]{hyperref}
\usepackage{color,comment}
\usepackage{algorithm,algorithmic}
\usepackage[font=small,format=plain,labelfont=bf,up]{caption}
\usepackage{lpic}
\usepackage{subfigure}
\usepackage{palatino,mathpazo}
\usepackage{multirow}
\usepackage{graphics}
\usepackage{setspace}
\usepackage{xypic}

\def\etal{et al.\ }
\usepackage{algorithm}

\newcommand{\beq}{\begin{equation}}
\newcommand{\eeq}{\end{equation}}
\newcommand{\bea}{\begin{eqnarray*}}
\newcommand{\eea}{\end{eqnarray*}}

\newcommand{\mycomment}[1]{}

\def\Z{{\mathbb Z}}
\newcommand{\cc}[2]{\overline{C}({#1}, #2)}

\newtheorem{lemma}{Lemma}
\newtheorem{theorem}[lemma]{Theorem}

\newtheorem{conjecture}[lemma]{Conjecture}
\newtheorem{claim}[lemma]{Claim}
\newtheorem{subclaim}{Claim}[lemma]

\theoremstyle{definition}
\newtheorem{definition}[lemma]{Definition}
\newtheorem{observation}[lemma]{Observation}

\newtheorem{construction}{Construction}

\input colordvi
 \newenvironment{cit}
{
    \begin{list}{- \ }{}
        \setlength{\topsep}{0pt}
        \setlength{\parskip}{0pt}
        \setlength{\partopsep}{0pt}
        \setlength{\parsep}{0pt}         
        \setlength{\itemsep}{0pt} 
}
{
    \end{list} 
}
 \newenvironment{cem}
{
    \begin{enumerate}
        \setlength{\topsep}{0pt}
        \setlength{\parskip}{0pt}
        \setlength{\partopsep}{0pt}
        \setlength{\parsep}{0pt}         
        \setlength{\itemsep}{0pt} 
}
{
    \end{enumerate} 
}
\usepackage{fancyhdr}

\newcounter{casenum}	
\newcounter{subcasenum}	
\numberwithin{subcasenum}{casenum}
\newcounter{subsubcasenum}	
\numberwithin{subsubcasenum}{subcasenum}

\renewcommand{\thecasenum}{\arabic{casenum}}
\renewcommand{\thesubcasenum}{\thecasenum.\roman{subcasenum}}

\newcommand{\theSuperClaim}{\arabic{lemma}}

\newcounter{stagenum}
\renewcommand{\thestagenum}{1$\alpha$}

\newenvironment{mycases}
{
  \list{}{%
    \leftmargin0.5cm   
    \rightmargin0cm
  }
  \item\relax
	\setcounter{casenum}{0}
}
{	
	\endlist
}
\newenvironment{subcases}
{
  \list{}{%
    \leftmargin0.5cm   
    \rightmargin0cm
  }
  \item\relax
}
{	
	\endlist
}

\newenvironment{casefig}
{
\vspace{0em}
	\begin{figure}[H]
		\centering
}
{
	\end{figure}	
\vspace{-1em}
}

\def\casefigratio{1}

\newcommand{\case}[1]{
	\vspace{0.5em}
	
	\refstepcounter{casenum}
	\noindent\hspace{-0.5cm}\textit{Case \thecasenum: #1} 
}

\newcommand{\subcase}[1]{
	\vspace{0.25em}
	
	\refstepcounter{subcasenum}
	\noindent\hspace{-0.5cm}\textit{Case \thesubcasenum: #1} 
}

\newcommand{\subsubcaseitem}[0]{
	\refstepcounter{subsubcasenum}
	\noindent\hspace{-0.6cm}{(\alph{subsubcasenum})}
}

\begin{document}

\def\Stab{\operatorname{Stab}}
\def\cO{\mathcal{O}}
\def\cA{\mathcal{A}}
	
\title{Uniquely $K_r$-Saturated Graphs}
\author{Stephen G. Hartke$^{1,2}$ \and Derrick Stolee$^{1,2,3}$}
\date{\today}

\maketitle

\begin{abstract}
	A graph $G$ is \textit{uniquely $K_r$-saturated} if it contains no clique with $r$ vertices
		and if for all edges $e$ in the complement, $G+e$ has a unique clique with $r$ vertices.
    Previously, few examples of uniquely $K_r$-saturated graphs were known, and little was known about their properties.
	We search for these graphs by adapting orbital branching, 
		a technique originally developed for symmetric integer linear programs.	
	We find several new uniquely $K_r$-saturated graphs with $4 \leq r \leq 7$,
		as well as two new infinite families based on Cayley graphs for $\Z_n$ 
		with a small number of generators.
\end{abstract}

\footnotetext[1]{Department of Mathematics, University of Nebraska--Lincoln, Lincoln NE 68588-0130.\\ \texttt{$\{$hartke;s-dstolee1$\}$@math.unl.edu}}
\footnotetext[2]{Research supported in part by NSF Grant DMS-0914815.}
\footnotetext[3]{Department of Computer Science and Engineering, University of Nebraska--Lincoln, Lincoln NE 68588.}

\section{Introduction}

A graph $G$ is \textit{uniquely $H$-saturated} if there is no subgraph 
        of $G$ isomorphic to $H$,
		and for all edges $e$ in the complement of $G$ there is a unique subgraph in $G+e$ isomorphic to $H$\footnote[4]{%
	A technicality: for all $t < n(H)$, the complete graph $K_t$ is trivially uniquely $H$-saturated. We adopt the convention that always $n(G) \geq n(H)$.}.
Uniquely $H$-saturated graphs were introduced by Cooper, Lenz, LeSaulnier, Wenger, and West~\cite{CLLWW}
	where they classified uniquely $C_k$-saturated graphs for $k \in \{3,4\}$;
	in each case there is a finite number of graphs.
Wenger~\cite{Wenger,WengerC8} classified the uniquely $C_5$-saturated graphs 
	and proved that there do not exist any uniquely $C_k$-saturated graphs for $k \in \{6,7,8\}$.

In this paper, we focus on the case where $H = K_r$, the complete graph of order $r$.
Usually $K_r$ is the first graph considered for extremal and saturation problems.
However, we find that classifying all uniquely $K_r$-saturated graphs 
	is far from trivial, 
    even in the case that $r = 4$.

Previously, few examples of uniquely $K_r$-saturated graphs were known, and little was known about their properties. 
We adapt the computational technique of orbital branching into the graph theory setting to search for uniquely $K_r$-saturated graphs.  Orbital branching was originally introduced by Ostrowski, Linderoth, Rossi, and Smriglio~\cite{OrbitalBranching} to solve symmetric integer programs. 
We further extend the technique to use augmentations which are customized to this problem.
By executing this search, we found several new uniquely $K_r$-saturated graphs for $r \in \{4,5,6,7\}$
    and we provide constructions of these graphs to understand their structure.
One of the graphs we discovered is a Cayley graph, 
	which led us to design a search for Cayley graphs which are uniquely $K_r$-saturated.  Motivated by these search results, we construct two new infinite families of uniquely $K_r$-saturated Cayley graphs.

Erd\H{o}s, Hajnal, and Moon~\cite{ErdosHajnalMoon} studied the minimum number of edges in a $K_r$-saturated graph.  They proved that the only extremal examples are 
	the graphs formed by adding $r-2$ dominating vertices to an independent
	set; these graphs are also uniquely $K_r$-saturated.
However, if $G$ is uniquely $K_r$-saturated and has a dominating vertex, 
	then deleting that vertex results in a uniquely $K_{r-1}$-saturated graph.
To avoid the issue of dominating vertices, we define
	a graph to be $r$-\textit{primitive} if it is
	uniquely $K_r$-saturated and has no dominating vertex.
Understanding which $r$-primitive graphs exist
	is fundamental to characterizing uniquely $K_r$-saturated graphs.

Since $K_3 \cong C_3$, the uniquely $K_3$-saturated graphs were proven by Cooper {\etal}\cite{CLLWW}
	to be stars and Moore graphs of diameter two.
While stars are uniquely $K_3$-saturated, they are not $3$-primitive.
The Moore graphs of diameter two are exactly the 3-primitive graphs;
	Hoffman and Singleton~\cite{HoffmanSingleton} proved there are a finite number of these graphs.

David Collins and Bill Kay discovered the only previously known infinite family of
r-primitive graphs, that of complements of odd cycles: $\overline{C_{2r-1}}$ is $r$-primitive.
Collins and Cooper discovered two more $4$-primitive graphs of orders $10$ and $12$~\cite{CooperWenger}.
These two graphs are described in detail in Section \ref{sec:constructions}.

One feature of all previously known $r$-primitive graphs is that they are all regular.
Since proving regularity has been instrumental in previous characterization proofs~(such as \cite{CLLWW,HoffmanSingleton}), 
	there was a hope that $r$-primitive graphs are regular.
However, we present a counterexample: a $5$-primitive graph on $16$ vertices 
	with minimum degree 8 and maximum degree 9.

The major open question in this area concerns 
	the number of $r$-primitive graphs for a fixed $r$.

\begin{conjecture}
	For each $r \geq 3$, there are a finite number of $r$-primitive graphs.
\end{conjecture}

This conjecture is true for $r = 3$~\cite{HoffmanSingleton} and otherwise completely open.
Before this work, it was not even known if there was more than one $r$-primitive graph for any $r \geq 5$.
After we discovered the graphs in this work (which lack any common structure and sometimes appear very strange), 
	we are unsure the conjecture holds even for $r = 4$.

In Section \ref{sec:summary}, we briefly summarize our results, including our computational method, the new sporadic $r$-primitive graphs, and our new algebraic constructions.
%


\section{Summary of results}
\label{sec:summary}

Our results have three main components.
First, we develop a computational method for generating uniquely $K_r$-saturated graphs.
Then, based on one of the generated examples, we construct two new infinite families of uniquely $K_r$-saturated graphs.
Finally, we describe all known uniquely $K_r$-saturated graphs, 
	including the nine new sporadic\footnote{We call a graph \textit{sporadic} if it has not yet been extended to an infinite family. Therefore, even though our search found 10 new graphs, one extended to an infinite family and so is not sporadic.} graphs found using the computational method.

\subsection{Computational method}

In Section \ref{sec:krcompletions}, we develop a new technique for exhaustively searching for uniquely $K_r$-saturated graphs on $n$ vertices.
The search is based on the technique of \emph{orbital branching} originally developed for use in symmetric integer programs by Ostrowski, Linderoth, Rossi, and Smriglio~\cite{OrbitalBranching, ConstraintOrbitalBranching}.
We focus on the case of constraint systems with variables taking value in $\{0,1\}$.
The orbital branching is based on the standard branch-and-bound technique where an unassigned variable is selected and the search branches into cases for each possible value for that variable.
In a symmetric constraint system, the automorphisms of the variables which preserve the constraints and variable values generate orbits of variables.
Orbital branching selects an orbit of variables and branches in two cases.
The first branch selects an arbitrary representative variable is selected from the orbit and set to zero.
The second branch sets all variables in the orbit to one.

We extend this technique to be effective to search for uniquely $K_r$-saturated graphs.
We add an additional constraint to partial graphs: 
	if a pair $v_i, v_j$ is a non-edge in $G$, 
	then there is a unique set $S_{i,j}$ containing $r-2$ vertices 
	so that $S_{i,j}$ is a clique and every edge between $\{v_i,v_j\}$ and $S_{i,j}$
	is included in $G$.
This guarantees that there is at least one copy of $K_r$ in $G + v_iv_j$ for all assignments of edges and non-edges
	to the remaining unassigned pairs.
The orbital branching method is customized to enforce this constraint,
	which leads to multiple edges being added to the graph in 
	every augmentation step.
By executing this algorithm, we found 10 new $r$-primitive graphs.

\subsection{New $r$-primitive graphs}
	
For $r \in \{4,5,6,7,8\}$, we used this method to exhaustively search for uniquely $K_r$-saturated graphs of order at most $N_r$, where $N_4 = 20$, $N_5 = N_6 = 16$, and 
	$N_7 = N_8 = 17$.
Table \ref{tab:rprimitivegraphs} lists the $r$-primitive graphs that were discovered in this search.
Most graphs do not fit a short description and are labeled $G_{N}^{(i)}$, where $N$ is the number of vertices and $i \in \{ A, B, C \}$ distinguishes between graphs of the same order.

\begin{table}[htp]
	\centering
	\scalebox{0.9}{
	\renewcommand\arraystretch{1.5}
	\begin{tabular}{c||c|c|c|c|c|c}
	 $n$ &  13 & 15  & 16 & 16 & 17 & 18 \\
	 \hline
	 $r$ &  4 & 6  &  5  & 6 & 7  & 4\\
	 \hline&&&&&&\\[-3ex]
	 Graphs &  
	 	$G_{13}$, Paley$(13)$  & 
		$G_{15}^{(A)}, G_{15}^{(B)}$ & 
		$G_{16}^{(A)}, G_{16}^{(B)}$ & 
		$G_{16}^{(C)}$ &
		$\cc{\Z_{17}}{\{1,4\}}$ &
		$G_{18}^{(A)}, G_{18}^{(B)}$
	\end{tabular}}
	
	\caption{\label{tab:rprimitivegraphs} Newly discovered $r$-primitive graphs.}
\end{table}

In all, ten new graphs were discovered to be uniquely $K_r$-saturated by this search.
Explicit constructions of these graphs are given in Section \ref{sec:constructions}.
Two graphs found by computer search are vertex-transitive and have a prime number of vertices.
Observe that vertex-transitive graphs with a prime number of vertices are Cayley graphs.
One vertex-transitive 4-primitive graph is the Paley graph
	of order 13 (see \cite{Paley}).
The other vertex-transitive graph is 7-primitive 
	on 17 vertices and is 14 regular.
However, it is easier to understand its complement, which is the Cayley graph for $\Z_{17}$ generated by $1$ and $4$.
This graph is listed as $\cc{\Z_{17}}{\{1,4\}}$ in Table \ref{tab:rprimitivegraphs}
	and is the first example of our new infinite families, described below.

\subsection{Algebraic Constructions}

For a finite group $\Gamma$ and a generating set $S \subseteq \Gamma$, 
	let $C(\Gamma, S)$ be the \emph{Cayley graph} for $\Gamma$ generated by $S$:
	the vertex set is $\Gamma$ and two elements $x, y \in \Gamma$ are adjacent
	if and only if there is a $z \in S$ where $x = yz$ or $x = yz^{-1}$.
When $\Gamma \cong \Z_n$, the resulting graph is also called a \emph{circulant graph}.
The cycle $C_n$ can be described as the Cayley graph of $\Z_n$ generated by $1$.
Since $\overline{C_{2r-1}}$ is $r$-primitive and
	we discovered a graph on $17$ vertices whose complement is a Cayley graph
	with two generators, 
	we searched for $r$-primitive graphs when restricted to complements of Cayley graphs
	with a small number of generators.

For a finite group $\Gamma$ and a set $S \subseteq \Gamma$,
	the \emph{Cayley complement} $\overline{C}(\Gamma,S)$ is the complement
	of the Cayley graph $C(\Gamma, S)$.
We restrict to the case when $\Gamma = \Z_n$ for some $n$,
	and the use of the complement allows us to use a small number of 
	generators while generating dense graphs.

We search for $r$-primitive Cayley complements by enumerating all small generator sets $S$, then iterate over $n$ where $n \geq 2\max S + 1$ and build $\cc{\Z_n}{S}$.
If $\cc{\Z_n}{S}$ is $r$-primitive for any $r$, it must be for $r = \omega(\cc{\Z_n}{S}) + 1$,
	so we compute this $r$ using Niskanen and \"Osterg\r{a}rd's \emph{cliquer} library~\cite{cliquer}.
Also using \emph{cliquer}, we count the number of $r$-cliques in 
	$\cc{\Z_n}{S} + \{0,i\}$ for all $i \in S$.
Since $\cc{\Z_n}{S}$ is vertex-transitive, this provides sufficient information to determine if $\cc{\Z_n}{S}$ is $r$-primitive.
The successful parameters for $r$-primitive Cayley complements
	with $g$ generators
	are given in Tables \ref{tab:twogens} ($g=2$),
	\ref{tab:threegens} ($g=3$), and
	\ref{tab:moregens} ($g\geq 4$).

\def\linegap{0.4pt}
\begin{table}[tp]
	\centering
	\mbox{
	\subfigure[\label{tab:twogens}Two Generators]{
	\begin{tabular}[H]{cccc}
	$t$ & $S$ & $r$ & $n$\\
	\hline
		\hline
2 & $\{ 1,4\}$ & 7 & 17 \\[\linegap]
3 & $\{ 1,6\}$ & 16 & 37 \\[\linegap]
4 & $\{ 1,8\}$ & 29 & 65 \\[\linegap]
5 & $\{ 1,10\}$ & 46 & 101 \\[\linegap]
6 & $\{ 1,12\}$ & 67 & 145 
	\end{tabular}
	}
	\subfigure[\label{tab:threegens}Three Generators]{
		\begin{tabular}[H]{cccc}
			$t$ & $S$ & $r$ & $n$\\
			\hline
		\hline
2 & $\{1, 5, 6\}$ & 9 & 31\\[\linegap]
3 & $\{1, 8, 9\}$ & 22 & 73\\[\linegap]
4 & $\{1, 11, 12\}$ & 41 & 133\\[\linegap]
5 & $\{1, 14, 15\}$ & 66 & 211\\[\linegap]
6 & $\{1, 17, 18\}$ & 97 & 307 
		\end{tabular}
	}
	}
	\subfigure[\label{tab:moregens}Sporadic Cayley Complements]{
	\begin{tabular}[H]{cccc}
		$g$ & $S$ & $r$ & $n$  \\
		\hline
		\hline
		3 & $\{1, 3, 4\}$ & $4$ & $13$ \\
		\hline
		\multirow{2}{*}{$4$} 
			& $\{  1,   5,   8,  34 \}$ 
			& \multirow{2}{*}{$28$}
			&  \multirow{2}{*}{$89$}  \\
		    &$\{ 1, 11, 18, 34 \}$ &&\\
		\hline
		$5$ & $\{1, 5, 14, 17, 25 \}$  & $19$& $71$ \\
		\hline
		$5$ & $\{1, 6, 14, 17, 36 \}$ & $27$ & $101$ \\
		\hline
		$6$ & $\{   1,   6,  16,  22,  35,  36 \}$ & $21$ & $97$ \\
		\hline
		$6$ & $\{   1,   8,  23,  26,  43,  64 \}$ & $54$ & $185$ \\
		\hline
		$7$ & $\{   1,  20,  23,  26,  30,  32,  34 \}$ & $15$  & $71$\\
		\hline
		$8$ & $\{ 1,  8,  12,  18,  22,  27,  33,  47 \}$ & $20$ & $97$ \\
		\hline
		$9$ & $\{   1,   4,  10,  16,  25,  27,  33,  40,  64 \}$ & $28$ & $133$  
	\end{tabular}
	}
	\caption{\label{tab:largecayley}Cayley complement parameters for $r$-primitive graphs over $\Z_n$.}
\end{table}

For two and three generators, a pattern emerged in the generating sets
	and interpolating the values of $n$ and $r$ resulted 
	in two infinite families of $r$-primitive graphs:

\begin{theorem}
	\label{thm:twogenexample}
	Let $t \geq 2$ and set $n = 4t^2+1, r = 2t^2-t+1$.
	Then, $\cc{\Z_n}{\{1,2t\}}$ is $r$-primitive.
\end{theorem}

\begin{theorem}
	\label{thm:threegenerators}
	Let $t \geq 2$ and set $n = 9t^2-3t+1, r = 3t^2-2t+1$.
	Then, $\cc{\Z_n}{\{1,3t-1, 3t\}}$ is $r$-primitive.
\end{theorem}

An important step to proving these Cayley complements are $r$-primitive
	is to compute the clique number.
Computing the clique number or independence number of a Cayley graph 
	is very difficult, as many papers study this question~\cite{GreenCayley,KSUnitaryCayley}, 
	including in the special cases of circulant graphs~\cite{BICliqueCirculant,BHFractionalRamsey,HoshinoCirculant,XXYCirculant}
	and Paley graphs~\cite{BEHWPaleySquares,BlokhuisPaleySquares,BDRPaley,CohenPaley}.
Our enumerative approach to Theorem~\ref{thm:twogenexample} and discharging approach to
	Theorem~\ref{thm:threegenerators} provide a new perspective on computing these values.

It remains an open question if an infinite family of Cayley complements $\cc{\Z_n}{S}$ exist
	for a fixed number of generators $g = |S|$ where $g \geq 4$.
For all known constructions with $g \neq 4$, 
	observe that the generators are roots of unity in $\Z_n$
	with $x^{2g} \equiv 1 \pmod n$ for each generator $x$.
Being roots of unity is not a sufficient condition for the Cayley complement to be $r$-primitive,
	but this observation may lead to algebraic techniques to build more infinite families of Cayley complements.  

Determining the maximum density of a clique and independent set for infinite Cayley graphs (i.e., $\cc{\Z}{S}$, where $S$ is finite) would be useful for providing bounds on the finite graphs.
Further, such bounds could be used by algorithms to find and count large cliques and independent sets in finite Cayley graphs.


\section{Orbital branching using custom augmentations}
\label{sec:krcompletions}
In this section, we describe a computational method to search for uniquely $K_r$-saturated graphs.
We shall build graphs piece-by-piece by selecting pairs of vertices to be edges or non-edges.

To store partial graphs, we use the notion of a \emph{trigraph},
	defined by Chudnovsky \cite{trigraphs} and used by Martin and Smith \cite{InducedSaturation}.
A \emph{trigraph} $T$ is a set of $n$ vertices $v_1, \dots, v_n$ where every pair $v_iv_j$ is colored black, white, or gray.
The black pairs represent edges, the white edges represent non-edges, and the gray edges are unassigned pairs.
A graph $G$ is a \emph{realization} of a trigraph $T$ if all black pairs of $T$ are edges of $G$ and all white pairs of $T$ are non-edges of $G$. 
Essentially, a realization is formed by assigning the gray pairs to be edges or non-edges.
In this way, we consider a graph to be a trigraph with no gray pairs.

Non-edges play a crucial role in the structure of uniquely $K_r$-saturated graphs.
Given a trigraph $T$ and a pair $v_iv_j$, a set $S$ of $r-2$ vertices is a \emph{$K_r$-completion} for $v_iv_j$ if every pair in $S \cup \{v_i, v_j\}$ is a black edge, except for possibly $v_iv_j$.
Observe that a $K_r$-free graph is uniquely $K_r$-saturated if and only if every non-edge has a unique $K_r$-completion.

We begin with a completely gray trigraph and build uniquely $K_r$-saturated graphs by adding black and white pairs.
If we can detect that no realization of the current trigraph can be uniquely $K_r$-saturated, then we backtrack and attempt a different augmentation.
The first two constraints we place on a trigraph $T$ are:

\begin{cem}
	\item[(C1)] There is no black $r$-clique in $T$.
	
	\item[(C2)] Every vertex pair has at most one black $K_r$-completion.
\end{cem}

It is clear that a trigraph failing either of these conditions will fail to have a uniquely $K_r$-saturated realization.

We use the symmetry of trigraphs to reduce the number of isomorphic duplicates.
The \textit{automorphism group} of a trigraph $T$ is the set of permutations of the vertices that preserve the colors of the pairs.
These automorphisms are computed with McKay's \textit{nauty} library \cite{HRnauty, nauty} through the standard method of using a layered graph.
	
\def\cO{\mathcal{O}}
\def\cA{\mathcal{A}}

\subsection{Orbital Branching}

Ostrowski, Linderoth, Rossi, and Smriglio introduced the 
	technique of \emph{orbital branching} for symmetric integer programs with 0-1 variables \cite{OrbitalBranching} 
	and for symmetric constraint systems \cite{ConstraintOrbitalBranching}.
Orbital branching extends the standard branch-and-bound strategy of combinatorial optimization by exploiting symmetry to reduce the search space.
We adapt this technique to search for graphs by using trigraphs in place of variable assignments.

Given a trigraph $T$, compute the automorphism group and select an \emph{orbit} $\cO$ of gray pairs.
Since every representative pair in $\cO$ is identical in the current trigraph, assigning any representative to be a white pair leads to isomorphic trigraphs.
Hence, we need only attempt assigning a single pair in $\cO$ to be white.
The natural complement of this operation is to assign all pairs in $\cO$ to be black.
Therefore, we branch on the following two options:
	
\begin{cit}
	\item \emph{Branch 1}: Select any pair in $\cO$ and assign it the color white.
	\item \emph{Branch 2}: Assign all pairs in $\cO$ the color black.
\end{cit}

A visual representation of this branching process is presented in 
	Figure \ref{subfig:orbbranching}.
	
An important part of this strategy is to select an appropriate orbit.
The selection should attempt to maximize the size of the orbit 
	(in order to exploit the number of pairs assigned in the second branch)
	while preserving as much symmetry as possible
	(in order to maintain large orbits in deeper stages of the search).
It is difficult to determine the appropriate branching rule \emph{a priori},
	so it is beneficial to implement and compare the performance of several branching rules.
	
This use of orbital branching suffices to create a complete search of all uniquely $K_r$-saturated graphs, but is not very efficient.
One significant drawback to this technique is the fact that the constraints (C1) and (C2) rely on black pairs forming cliques.
In the next section, we create a custom augmentation step that is aimed at making these constraints trigger more frequently and thereby reducing the number of generated trigraphs.

\subsection{Custom augmentations}

We search for uniquely $K_r$-saturated graphs
	by enforcing at each step that every white pair
	has a unique $K_r$-completion.
We place the following constraints on a trigraph: 

\begin{cem}
	\item[(C3)] If $v_iv_j$ is a white edge, then there exists a unique $K_r$-completion $S \subseteq \{v_1,\dots,v_n\}$ for $v_iv_j$.
\end{cem}

To enforce the constraint (C3), whenever we assign a white pair we shall also select a set of $r-2$ vertices to be the $K_r$-completion and assign the appropriate pairs to be black.
The orbital branching procedure was built to assign only one white pair in a given step, so we can attempt all possible $K_r$-completions for that pair.
However, if we perform an automorphism calculation and only augment for one representative set from every orbit of these sets, we can reduce the number of isomorphic duplicates.

We follow a two-stage orbital branching procedure.
In the first stage, we select an orbit $\cO$ of gray pairs.
Either we select a representative pair $v_{i'}v_{j'} \in \cO$ to set to white or assign $v_iv_j$ to be black for all pairs $v_iv_j \in \cO$.
In order to guarantee constraint (C3), the white pair must have a $K_r$-completion.
We perform a second automorphism computation to find $\Stab_{\{v_{i'}, v_{j'}\}}(T)$, the set of automorphisms which set-wise stabilize the pair $v_{i'}v_{j'}$.
Then, we compute all orbits of $(r-2)$-subsets $S$ in $\{v_1,\dots,v_n\} \setminus \{v_i,v_j\}$ under the action of $\Stab_{\{v_{i'}, v_{j'}\}}(T)$.
The second stage branches on each set-orbit $\cA$, selects a single representative $S' \in \cA$
	and adds all necessary black pairs to make $S'$ be a $K_r$-completion for $v_{i'}v_{j'}$.
If at any point we attempt to assign a white pair to be black, that branch fails and we continue with the next set-orbit.	

This branching process on a trigraph $T$ is:
	
\begin{cit}
	\item \emph{Branch 1}: Select any pair $v_{i_1}v_{j_1} \in \cO$ to be white.
	\begin{cit}
		\item \emph{Sub-Branch}: For \emph{every} orbit $\cA$ of $(r-2)$-subsets of $V(T)\setminus \{v_{i_1},v_{i_2}\}$ under the action of $\Stab_{\{v_{i_1},v_{j_1}\}}(T)$, select any set $S \in \cA$, assign $v_{i_1}v_a$, $v_{j_1}v_a$, and $v_{a}v_{b}$ to be black for all $v_a, v_b \in S$.
	\end{cit}
	\item \emph{Branch 2}: Set $v_iv_j$ to be black for all pairs $v_iv_j \in \cO$.
\end{cit}
	
\begin{figure}[htbp]
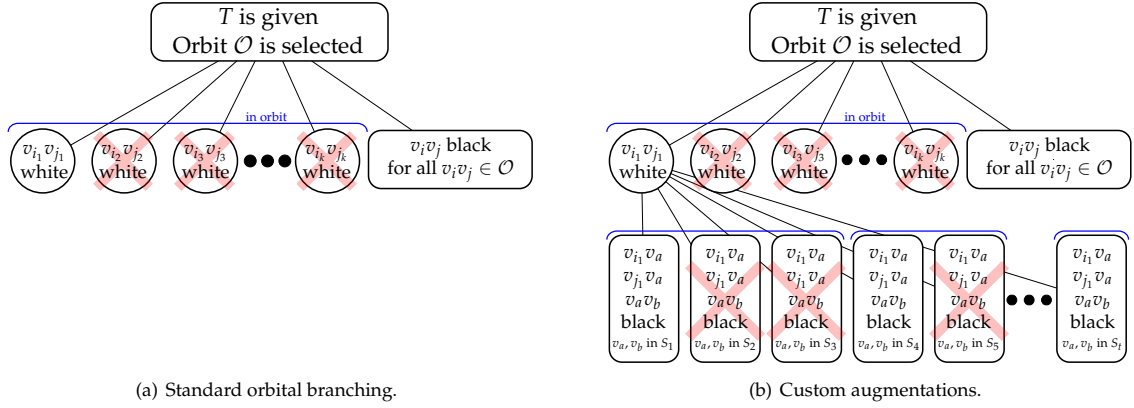

	\centering
	\mbox{
		\scalebox{0.8}{
		\subfigure[\label{subfig:orbbranching}Standard orbital branching.]{
			\begin{lpic}[]{"figs-unique/OrbBranchingLPic"(3.5in,)}
				\lbl[]{140,175;\parbox{2in}{\centering $T$ is given\\Orbit $\cO$ is selected}}
				\lbl[b]{139,128;\tiny\color{blue}{in orbit}}
				\lbl[]{235,108;\parbox{1in}{\footnotesize\centering $v_iv_j$ black\\ for all $v_iv_j \in \cO$}}
				\lbl[]{23.5,107;\footnotesize\parbox{0.4in}{\centering$v_{i_1}v_{j_1}$\\white}}
				\lbl[]{64.5,107;\footnotesize\parbox{0.4in}{\centering$v_{i_2}v_{j_2}$\\white}}
				\lbl[]{107.5,107;\footnotesize\parbox{0.4in}{\centering$v_{i_3}v_{j_3}$\\white}}
				\lbl[]{170.5,107;\footnotesize\parbox{0.4in}{\centering$v_{i_k}v_{j_k}$\\white}}
			\end{lpic}
		}
		\qquad
		\subfigure[\label{subfig:orbbranchingcustom}Custom augmentations.]{
			\begin{lpic}[]{"figs-unique/OrbBranchingCustomLPic"(3.5in,)}
				\lbl[]{140,175;\parbox{2in}{\centering $T$ is given\\Orbit $\cO$ is selected}}
				\lbl[b]{139,128;\tiny\color{blue}{in orbit}}
				\lbl[]{235,108;\parbox{1in}{\footnotesize\centering $v_iv_j$ black\\ for all $v_iv_j \in \cO$}}
				\lbl[]{23.5,107;\footnotesize\parbox{0.4in}{\centering$v_{i_1}v_{j_1}$\\white}}
				\lbl[]{64.5,107;\footnotesize\parbox{0.4in}{\centering$v_{i_2}v_{j_2}$\\white}}
				\lbl[]{107.5,107;\footnotesize\parbox{0.4in}{\centering$v_{i_3}v_{j_3}$\\white}}
				\lbl[]{170.5,107;\footnotesize\parbox{0.4in}{\centering$v_{i_k}v_{j_k}$\\white}}
				\lbl[]{23.75,35;\footnotesize\parbox{0.5in}{\centering$v_{i_1}v_a$\\$v_{j_1}v_a$\\$v_{a}v_b$\\black\\\vspace{0.25em}\tiny$v_a,v_b$ in $S_1$}}
				\lbl[]{66.5,35;\footnotesize\parbox{0.5in}{\centering$v_{i_1}v_a$\\$v_{j_1}v_a$\\$v_{a}v_b$\\black\\\vspace{0.25em}\tiny$v_a,v_b$ in $S_2$}}
				\lbl[]{108.5,35;\footnotesize\parbox{0.5in}{\centering$v_{i_1}v_a$\\$v_{j_1}v_a$\\$v_{a}v_b$\\black\\\vspace{0.25em}\tiny$v_a,v_b$ in $S_3$}}
				\lbl[]{151.5,35;\footnotesize\parbox{0.5in}{\centering$v_{i_1}v_a$\\$v_{j_1}v_a$\\$v_{a}v_b$\\black\\\vspace{0.25em}\tiny$v_a,v_b$ in $S_4$}}
				\lbl[]{193,35;\footnotesize\parbox{0.5in}{\centering$v_{i_1}v_a$\\$v_{j_1}v_a$\\$v_{a}v_b$\\black\\\vspace{0.25em}\tiny$v_a,v_b$ in $S_5$}}
				\lbl[]{257.5,35;\footnotesize\parbox{0.5in}{\centering$v_{i_1}v_a$\\$v_{j_1}v_a$\\$v_{a}v_b$\\black\\\vspace{0.25em}\tiny$v_a,v_b$ in $S_t$}}
			\end{lpic}
		}
		}
	}
	\caption{\label{fig:exampleorbbranching}Visual description of the branching process.}
\end{figure}

\begin{algorithm}[tp]
	\caption{\label{alg:orbbranch} SaturatedSearch($n, r, T$)}
	\begin{algorithmic}
		\IF{$T$ contains a black $r$-clique}
			\STATE \textit{Constraint (C1) fails.}
			\RETURN
		\ELSIF{there exists a pair $v_iv_j$ with two $K_r$-completions in $T$}
			\STATE \textit{Constraint (C2) fails.}
			\RETURN
		\ELSIF{there are no gray pairs}
			\STATE \textit{The trigraph $T$ is uniquely $K_r$-saturated.}
			\STATE Output $T$.
			\RETURN
		\ENDIF
		\STATE \textit{Propagate under constraint (C1).}
		\FOR{all gray pairs $v_iv_j$}
			\IF{$v_iv_j$ has a $K_r$-completion in $T$}
				\STATE \textit{Assign $v_iv_j$ to be white.}
			\ENDIF
		\ENDFOR
		\STATE Compute pair orbits $\cO_1,\cO_2,\dots,$ of gray pairs $\{i,j\}$.
		\STATE Select an orbit $\cO_k$ using the branching rule.
		\STATE \textit{Branch 1.}
		\STATE Let $v_{i'}v_{j'}$ be a representative of $\cO_k$.
		\STATE Compute orbits $\cA_1,\cA_2,\dots, \cA_{\ell}$ of  $(r-2)$-vertex sets in $\{v_1,\dots,v_n\} \setminus \{v_{i'}, v_{j'}\}$.
		\FOR{$t \in \{1,\dots, \ell\}$}
			\STATE Let $S$ be a representative of $\cA_t$.
			\IF{$v_{i'}v_a, v_{j'}v_a, v_{a}v_b$ not white for all $a, b \in S$}
				\STATE \textit{Sub-Branch: Create $T'$ from $T$ by assigning $v_{i'}v_a, v_{j'}v_a, v_{a}v_b$ to be black for all $a, b \in S$.}
				\STATE \textbf{call} SaturatedSearch($n, r, T'$)
			\ENDIF
		\ENDFOR
		\STATE \textit{Branch 2: Create $T''$ from $T$ by assigning $v_iv_j$ to be black for all $v_iv_j \in \cO_k$.}
		\STATE \textbf{call} SaturatedSearch($n, r, T''$)
		\RETURN
	\end{algorithmic}
\end{algorithm}
	
The full algorithm to output all uniquely $K_r$-saturated graphs on 
	$n$ vertices is given as the recursive method 
	SaturatedSearch($n, r, T$) in Algorithm~\ref{alg:orbbranch},
	while the branching procedure is represented in Figure \ref{subfig:orbbranchingcustom}.
The algorithm is initialized using the trigraph corresponding to a single white pair with a $K_r$-completion.
The first step of every recursive call to SaturatedSearch($n, r, T$) is to
	 verify the constraints (C1) and (C2).
If either constraint fails, no realization of the current trigraph can be uniquely $K_r$-saturated, so we return.
After verifying the constraints, we perform a simple propagation step:
If a gray pair $\{i,j\}$ has a $K_r$-completion
	we assign that pair to be white.
We can assume that this pair is a white edge in order to avoid violation of (C1),
	and this assignment satisfies (C3).

The missing component of this algorithm is the \emph{branching rule}: the algorithm that 
	selects the orbit of unassigned pairs to use in the first stage of the branch.
Based on experimentation, the most efficient branching rule we implemented 
	only considers pairs where both vertices are contained in assigned pairs (if they exist)
	or pairs where one vertex is contained in an assigned pair (which must exist, otherwise),
	and selects from these pairs the orbit of largest size.
This choice would guarantee the branching orbit 
	has maximum interaction with currently assigned edges 
	while maximizing the effect of assigning all 
	representatives to be edges in the second branch.

\subsection{Implementation, Timing, and Results}

The full implementation is available as the \emph{Saturation} project in the \emph{SearchLib} software library\footnote{\emph{SearchLib} is available online at \url{http://www.math.unl.edu/~s-dstolee1/SearchLib/}}.
More information for the implementation is given in the \emph{Saturation} User Guide, 
	available with the software.
In particular, the user guide details the methods for verifying the constraints (C1), (C2), and (C3).
When $r \in \{4, 5\}$, we monitored clique growth using a custom data structure,
	but when $r \geq 6$ an implementation using 
	Niskanen and \"Osterg\r{a}rd's \textit{cliquer} library \cite{cliquer}
	was more efficient.

\begin{table}[btp]
	\begin{center}
	\begin{tabular}[tp]{c|r@{.}l|r@{.}l|r@{.}l|r@{.}l|r@{.}l}
		\multicolumn{1}{c|}{$n$} 
			 & \multicolumn{2}{c|}{$r = 4$} 
			 & \multicolumn{2}{c|}{$r = 5$} 
			 & \multicolumn{2}{c|}{$r = 6$} 
			 & \multicolumn{2}{c|}{$r = 7$}
			 & \multicolumn{2}{c}{$r = 8$}\\
		\hline
			10 & 0&10 s 
				& 0&37 s
				& 0&13 s
				& 0&01 s
				& 0&01 s\\
			11 & 0&68 s 
				& 5&25 s
				& 1&91 s
				& 0&28 s
				& 0&09 s\\
			12 & 4&58 s 
				&  1&60 m
				& 25&39 s
				&  1&97 s
				&  1&12 s\\
			13 & 34&66 s 
				& 34&54 m
				&  6&53 m
				& 59&94 s
				& 20&03 s \\
			14 & 4&93 m 
				& 10&39 h
				&  5&13 h
				& 20&66 m
				&  2&71 m\\
			15  & 40&59 m 
				& 23&49 d
				& 10&08 d
				& 12&28 h
				&  1&22 h\\
			16 & 6&34 h 
				& 1&58 y
				& 1&74 y
				& 34&53 d
				& 1&88 d\\
			17 & 3&44 d 
				& \multicolumn{2}{c|}{}
				& \multicolumn{2}{c|}{}
				&  8&76 y
				& 115&69 d\\
			18 & 53&01 d 
				& \multicolumn{2}{c|}{}
				& \multicolumn{2}{c|}{}
				& \multicolumn{2}{c|}{}
				& \multicolumn{2}{c}{}\\
			19 & 2&01 y
				& \multicolumn{2}{c|}{}
				& \multicolumn{2}{c|}{}
				& \multicolumn{2}{c|}{}
				& \multicolumn{2}{c}{}\\
			20 & 45&11 y 
				& \multicolumn{2}{c|}{}
				& \multicolumn{2}{c|}{}
				& \multicolumn{2}{c|}{}
				& \multicolumn{2}{c}{}\\
	\end{tabular}
	\end{center}

	\caption[CPU Times for the uniquely $K_r$-saturated graph search.]{\label{tab:UniqueKrTime}CPU times to search for uniquely $K_r$-saturated graphs of order $n$. Execution times from the Open Science Grid \cite{OpenScienceGrid} using the University of Nebraska Campus Grid \cite{WeitzelMS}.
The nodes available on the University of Nebraska Campus Grid
	consist of Xeon and Opteron processors with 
	a range of speed between 2.0 and 2.8 GHz.}
\end{table}

Our computational method is implemented using the \emph{TreeSearch} library \cite{TreeSearch}, which abstracts the search structure to allow for parallelization to a cluster or grid.
Table \ref{tab:UniqueKrTime} lists the CPU time taken by the search 
	for each $r \in \{4,5,6,7,8\}$ and $10 \leq n \leq N_r$ 
	(where $N_4 = 20$, $N_5 = N_6 = 16$, and $N_7 = N_8 = 17$) 
	until the search became intractable for $n = N_r + 1$.
Table \ref{tab:rprimitivegraphs} lists the $r$-primitive graphs of these sizes. 
Constructions for the graphs are given in Section \ref{sec:constructions}.

\section{Infinite families of $r$-primitive graphs using Cayley graphs}
\label{sec:cayleycomplements}

\def\cH{{\mathcal H}}

In this section, we prove Theorems \ref{thm:twogenexample} and \ref{thm:threegenerators}, which provide our two new infinite families of $r$-primitive graphs.
We begin with some definitions that are common to both proofs.

Fix an integer $n$, a generator set $S \subseteq \Z_n$, and a Cayley complement $G = \cc{\Z_n}{S}$.
For a set $X \subseteq \Z_n$ with $r = |X|$, 
	list the elements of $X$ as 
	$0 \leq x_0 \leq x_1 \leq \dots \leq x_{r-1} < n$.
We shall assume that $X$ is a clique in $G$ (or in $G + e$ for some nonedge 
	$e \in E(\overline{G})$).

Considering $X$ as a subset of ${\mathbb Z}_n$, we let the $k$th \emph{block} $B_k$ be the elements of $\Z_n$
	increasing from $x_k$ (inclusive) to $x_{k+1}$ (exclusive): $B_k  = \{ x_k, x_k + 1, \dots, x_{k+1} - 1\}$.
Note that $|B_k| = x_{k+1}-x_k$; we call a block of size $s$ an \emph{$s$-block}.
\def\cF{\mathcal{F}}
For an integer $t \geq 1$ and $j \in \{0,\dots, r-1\}$, the $j$th \emph{frame} $F_j$ is the collection of $t$ consecutive blocks in increasing order starting from $B_j$: $F_j = \{ B_j, B_{j+1}, \dots, B_{j+\ell-1}\}$.
A \emph{frame family} is a collection $\cF$ of frames.

If $F$ is a frame (or any set of blocks), define $\sigma(F) = \sum_{B_j \in F} |B_j|$, the number of elements covered by the blocks in $F$.

\begin{observation}
	If $X$ is a clique in $\cc{\Z_n}{S}$ and $F$ 
		is a set of consecutive blocks in $X$, then $\sigma(F) \notin S$.
\end{observation}

\subsection{Two Generators}

\noindent\textbf{Theorem \ref{thm:twogenexample}.}\ 
\emph{Let $t \geq 1$, and set $n = 4t^2+1$, $r = 2t^2-t+1$.
	Then, $\cc{\Z_n}{\{1,2t\}}$ is $r$-primitive.
}

\begin{proof}
	Let $G = \cc{\Z_n}{\{1,2t\}}$.
	Note that $G$ is regular of degree $n - 5$.
	If $t = 1$, then $n = 5$, $G$ is an empty graph, and $r = 2$, and empty graphs are 2-primitive.
	Therefore, we consider $t \geq 2$.
	
	\begin{claim}\label{clm:2t1}
		For a clique $X$, every frame $F_j$ has at least one block of size at least three, and $\sigma(F_j) \geq 2t+1$.
	\end{claim}
	
	All blocks $B_j$ have at least two elements, since no pair of elements in $X$ may be consecutive in $\Z_n$, so $\sigma(F_j) \geq 2t$.
	If for all $B_k \in F_j$ the block length $|B_k|$ is exactly two, then
		$\sigma(F_j) = 2t \in S$.
	Hence, there is some $B_k \in F_j$ so that $|B_k| \geq 3$ 
		and $\sigma(F_j) \geq 2t + 1$.
		
	We now prove there is no $r$-clique in $G$.
	
	\begin{claim}\label{claim:twogenclique}
		$\omega(G) < r$.
	\end{claim}
	
	Suppose $X \subseteq {\mathbb Z}_n$ is a clique of order $r$ in $G$.
	Let $\cF$ be the frame family of all frames ($\cF = \{ F_j : j \in \{0,\dots,r-1\} \}$)
		and consider the sum $\sum_{j=0}^{r-1} \sigma(F_j)$.
	Using the bound $\sigma(F_j) \geq 2t + 1$, we have this sum is at least $(2t+1)r$.
	Each block length $|B_k|$ is counted in $t$ evaluations of $\sigma(F_j)$ (for $j \in \{ k - t + 1, k - t + 2,\dots, k\}$).
	This sum counts each element of $\Z_n$ exactly $t$ times, giving value $tn$.
	This gives $tn = \sum_{j=0}^{r-1} \sigma(F_j) \geq (2t+1)r$, but $tn = 4t^3 + t < 4t^3 + t + 1 = (2t+1)r$, a contradiction.
	Hence, $X$ does not exist, proving the claim.
	
	To prove unique saturation, we consider only the non-edge $\{0,1\}$ since
		$G$ is vertex-transitive and the map $x \mapsto -2tx$ is an automorphism of $G$
		mapping the edge $\{0, 2t\}$ to $\{0, -4t^2\} \equiv \{ 0, 1\} \pmod n$.

	\begin{claim}\label{claim:twogenunique}
		There is a unique $r$-clique in $G + \{0,1\}$.
	\end{claim}	
	
	We may assume $X = \{ 0, 1, x_2, \dots, x_{r-1}\}$ 
		is an $r$-clique in $G + \{0,1\}$.
	We use the frame family $\cF$ defined as
	\[ \cF = \left\{ F_{jt+1} : j \in \{0,\dots, 2t-2 \} \right\}. \]
	
	Note that $\cF$ contains $2t-1$ \emph{disjoint} frames containing disjoint blocks, 		
		and the block $B_0 = \{x_0\}$ is not contained in any frame within $\cF$.
	Hence, $n - 1 = \sum_{F \in \cF} \sigma(F)$.
	By Claim \ref{clm:2t1}, we know that every frame $F \in \cF$ has $\sigma(F) \geq 2t+1$.
	This lower bound gives $\sum_{F\in \cF} \sigma(F) \geq (2t+1)(2t-1) = n - 2$.
	Thus, considering $\sigma(F)$ as an integer variable for each $F \in \cF$, 
		all solutions to the integer program with constraints
		$\sigma(F) \geq 2t + 1$ and $\sum_{F\in \cF} \sigma(F) = n - 1$
		have $\sigma(F) = 2t+1$ for all $F \in \cF$ except a unique $F' \in \cF$ with $\sigma(F') = 2t+2$.
				
	The frame $F'$ has two possible ways to attain $\sigma(F') = 2t + 2$:
		(a) have two blocks of size three, or (b) have one block of size four.
	However, if $F'$ has a block of size four, 
		then there is a 2-block $B_j \in F'$ on one end of $F'$
		where $\sigma(F' \setminus \{B_j\}) = 2t \in S$, a contradiction.
	Thus, $F'$ has two blocks of size three. 
	In addition, if $F'$ has fewer than $t-2$ blocks of size two 
		between the two blocks of size three, then there is a pair $x, y \in X$ with $y = x + 2t$.
	Therefore, $F'$ has two blocks of size three and they are the first and last blocks of $F'$.
	
	This frame family demonstrates the following properties of $X$.
	First, there are exactly $2t$ blocks of size three ($2t-2$ frames have exactly one and $F'$ has exactly two).
	Second, there is no set of $t$ consecutive blocks of size two.
	Finally, no two blocks of size three have fewer than $t-2$ blocks of size two between them.
	
	Consider the position of a 3-block in the first frame, $F_1$.
	If there are two 3-blocks in $F_1$, 
		they appear as the first and last blocks in $F_1$, but then 
		the distance from $x_0$ to $x_{t-1}$ is $2t$, a contradiction.
	Since there is exactly one 3-block, $B_k$, in $F_1$, suppose $k < t$.
	Then the distance from $x_0$ to $x_{t-1}$ is $2t$.
	Hence, $B_t$ is the 3-block in $F_1$.
	By symmetry, there must be $t-1$ $2$-blocks between the $3$-block in $F_{(2t-2)t+1}$ and $x_0$.
	
	Let $B_{k_1}, B_{k_2}, \dots, B_{k_{2t}}$ be the $3$-blocks in $X$ with $k_1 < k_2 < \cdots < k_{2t}$.
	By the position of the 3-block in $F_1$, we have $k_1 = t$.
	By the position of the 3-block in $F_{(2t-2)t+1}$, we have $k_{2t} = (2t-2)t+1$.
	Since 3-blocks must be separated by at least $t-1$ 2-blocks, 
		$k_{j+1} - k_{j} \geq t - 1$ but since $k_{2t} = (2t-1)(t-1) + k_1$
		we must have equality: $k_{j+1} - k_j = t - 1$.
	Assuming $X$ is an $r$-clique, it is uniquely defined by these properties.
	Indeed all vertices of this set are adjacent.
\end{proof}

\subsection{Three Generators}

\noindent\textbf{Theorem \ref{thm:threegenerators}.}\ 
\emph{
	Let $t \geq 1$ and set $n = 9t^2-3t+1, r = 3t^2-2t+1$.
	Then, $\cc{\Z_n}{\{1,3t-1, 3t\}}$ is $r$-primitive.
}

\begin{proof}
	Let $G = \cc{\Z_n}{\{1,3t-1, 3t\}}$.
	Observe that $G$ is vertex-transitive and there are automorphisms mapping 
		$\{0, 3t - 1\}$ to $\{0, 1\}$ 
		or $\{0, 3t\}$ to $\{0,1\}$.
	 Thus, we only need to verify that $G$ has no $r$-clique and 
	 	$G + \{0,1\}$ has a unique $r$-clique.
	
	We prove that $G$ is $r$-primitive in three steps.
	First, we show that there is no $r$-clique in $G$ in Claim \ref{claim:threegenclique} using discharging.
	Second, assuming there are no 2-blocks in an $r$-clique of $G + \{0,1\}$,
		we prove in Claim \ref{claim:threegenunique}
		that there is a unique such clique.
	This proof uses a counting method
		similar to the proof of Claim \ref{claim:twogenunique}.	
	Finally, we show that any $r$-clique in $G + \{0,1\}$ cannot contain any 2-blocks.
	This step is broken into Claims \ref{claim:complicateddischarging} and \ref{claim:itIScomplicated}, both of which slightly modify the discharging method 
		from Claim \ref{claim:threegenclique}
		to handle the 1-block.
	Claim \ref{claim:itIScomplicated} requires a detailed case analysis.
	
	We use several figures to aid the proof.
	Figure \ref{fig:blockframekey} shows examples of common features from these figures.
	
\begin{casefig}
			\scalebox{\casefigratio}{\begin{lpic}[]{"figs-unique/FigureKey"(,20mm)}
			   \lbl[b]{23,18;\small Frame}
			   \lbl[b]{41,18;\small Element}
			   \lbl[b]{56,18;\small Block}
			   \lbl[br]{90,16.5; \footnotesize Possible Element}
			   \lbl[bl]{95,16.5;\footnotesize Forbidden Element}
			   \lbl[tl]{112,5;\small $\Z_n$}
			   \lbl[tr]{138,8;\footnotesize Increasing}
			   \lbl[tl]{2,8;\footnotesize Decreasing}
			   \lbl[r]{48.5,2;\footnotesize $\{3t-1,3t\}$}
			\end{lpic}}
	\caption{\label{fig:blockframekey}Key to later figures}
\end{casefig}
	
	We begin by showing some basic observations which are used frequently in the rest of the proof.
	These observations focus on interactions among blocks that are forced by the generators $3t-1$ and $3t$.
	In the observations below, we define functions $\varphi_s$ and $\psi_s$ which map $s$-blocks of $X$ to other blocks of $X$.
	Always, $\varphi_s$ maps blocks \emph{forward} ($\varphi_s(B_k)$ has higher index than $B_k$) while $\psi_s$ maps blocks \emph{backward} ($\psi_s(B_k)$ has lower index than $B_k$).

It is intuitive that a maximum size clique uses as many small blocks as possible, to increase the density of the clique within $G$.
However, Observation \ref{obs:twoblocks} shows that every 2-block
	induces a block of size at least five \emph{in both directions}.
	
\begin{casefig}
			\scalebox{\casefigratio}{\begin{lpic}[]{"figs-unique/Observation2Blocks"(,22.5mm)}
			   \lbl[t]{88,20;\small$\varphi_2$}
			   \lbl[t]{48,20;\small$\psi_2$}
				\lbl[]{68.25,10.25;\footnotesize$x_{j}$}
				\lbl[]{77.75,10.25;\tiny$x_{j+1}$}
				\lbl[b]{70,15;$B_{j}$}
				\lbl[b]{15,15;$\psi_2(B_j)$}
				\lbl[b]{123,15;$\varphi_2(B_j)$}
			\end{lpic}}
	\caption{Observation \ref{obs:twoblocks} and a 2-block $B_j$.}
\end{casefig}
	
\begin{observation}[2-blocks]\label{obs:twoblocks}
	Let $B_j$ be a 2-block, so $x_{j+1} = x_j + 2$.
	The elements $x_j$ and $x_{j+1}$ along with generators $3t-1$ and $3t$
		guarantee that the sets $\{ x_j + 3t - 1, x_j + 3t, x_j + 3t + 1, x_j + 3t + 2\}$ and $\{x_j - 3t, x_j - 3t + 1, x_j - 3t + 2, x_j - 3t + 3\}$ do not intersect $X$.
	Since these sets contain consecutive elements, 
		each set is contained within a single block of $X$.
	We will use $\varphi_2(B_j)$ to denote the block containing $x_j + 3t$
		and $\psi_2(B_j)$ to denote the block containing $x_j - 3t$.
	Both $\varphi_2(B_j)$ and $\psi_2(B_j)$ have size at least five.
\end{observation}

If in fact multiple 2-blocks induce the same big block, 
	Observation \ref{obs:bigblocks} implies the big block has even larger size.

\begin{figure}[ht]
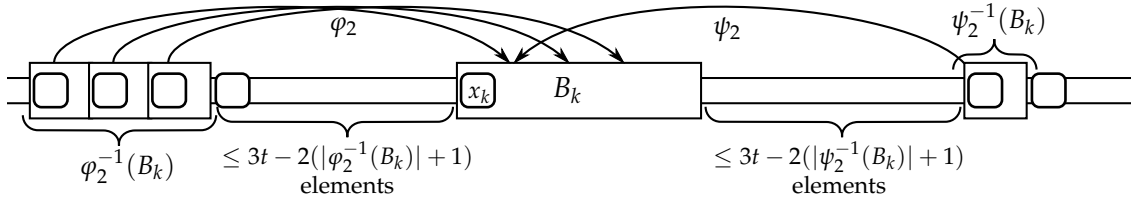

	\centering
			\scalebox{\casefigratio}{\begin{lpic}[]{"figs-unique/Observation5Blocks"(,22.5mm)}
			   \lbl[t]{90,19.5;\small$\psi_2$}
			   \lbl[t]{42,19.5;\small$\varphi_2$}
				\lbl[]{59,9.75;\footnotesize$x_{k}$}
				\lbl[t]{70,12;$B_{k}$}
				\lbl[t]{15,3;\small$\varphi_2^{-1}(B_k)$}
				\lbl[b]{124,17;\small$\psi_2^{-1}(B_k)$}
				\lbl[t]{42.5,4;\footnotesize\parbox{50mm}{\centering$\leq 3t - 2(|\varphi_2^{-1}(B_k)|+1)$\\ elements}}
				\lbl[t]{104,4; \footnotesize\parbox{50mm}{\centering$\leq 3t - 2(|\psi_2^{-1}(B_k)|+1)$\\ elements}}
			\end{lpic}}
	\caption{\label{fig:obsbigblock}Observation \ref{obs:bigblocks} and a block $B_k$.}
\end{figure}

\begin{observation}[Big blocks]\label{obs:bigblocks}
	Let $B_k$ be a block of size at least five.
	The set $\varphi_2^{-1}(B_k)$ is the set of 2-blocks $B_j$ so that $\varphi_2(B_j) = B_k$.
	Similarly, $\psi_2^{-1}(B_k)$ is the set of 2-blocks $B_j$ so that $\psi_2(B_j) = B_k$.
	Note that when $s = |\varphi_2^{-1}(B_k)|$, there are at least $s + 1$ elements of $X$
		($s$ from the 2-blocks in $\varphi_2^{-1}(B_k)$ and one following the last 2-block in $\varphi_2^{-1}(B_{k})$)
		which block $2(s+1)$ elements from containment in $X$ using the generators $3t - 1$ and $3t$.
	Therefore, 
	\[ |B_k| \geq 2|\varphi_2^{-1}(B_k)| + 3,
		\quad\text{and}\quad
		 |B_k| \geq 2|\psi_2^{-1}(B_k)| + 3.\]
	Further, there are at most $3t - 2(|\varphi_2^{-1}(B_k)| + 1)$ elements
		between $B_k$ and the last block of $\varphi_2^{-1}(B_k)$.
	Similarly, there are at most $3t - 2(|\psi_2^{-1}(B_k)|+1)$ elements between
		$B_k$ and the first block of $\psi_2^{-1}(B_k)$.
\end{observation}

\begin{casefig}
			\scalebox{\casefigratio}{\begin{lpic}[]{"figs-unique/Observation4Blocks"(,22.5mm)}
			   \lbl[t]{92,21;\small$\varphi_4$}
			   \lbl[t]{45,21;\small$\psi_4$}
			   \lbl[l]{102,16.5;\tiny\emph{or}}
			   \lbl[r]{39,16.5;\tiny\emph{or}}
				\lbl[]{63,10.25;\footnotesize$x_{j}$}
				\lbl[]{81.5,10.25;\tiny$x_{j+1}$}
				\lbl[b]{70,15;$B_{j}$}
				\lbl[t]{20,7.5;\parbox{1.5cm}{\tiny$x_j-3t+2$\\$x_j-3t+3$}}
				\lbl[t]{128,7.5;\parbox{1.5cm}{\tiny$x_j+3t+1$\\$x_j+3t+2$}}
			\end{lpic}}
	\caption{Observation \ref{obs:fourblocks} and a 4-block $B_j$.}
\end{casefig}

\begin{observation}[4-blocks]\label{obs:fourblocks}
	Let $B_j$ be a 4-block, so $x_{j+1} = x_j + 4$.
	The elements $\{x_j + 3t - 1, x_j + 3t, x_j +3t + 3, x_j + 3t + 4\}$
		are not contained in $X$,
		so $X \cap \{x_j + 3t - 1,\dots,x_j + 3t + 4\} \subseteq \{x_j + 3t+1, x_j + 3t +2\}$.
	In $G$, no two elements of $X$ are consecutive elements of $\Z_n$,
		so there is at most one element in this range.
	If there is no element of $X$ in $\{x_j + 3t +1, x_j + 3t + 2\}$,
		then there is a block of size at least seven that contains $x_j + 3t+ 1$.
	Otherwise, there is a single element in $X \cap \{x_j + 3t + 1, x_j + 3t+ 2\}$
		and one of the adjacent blocks has size at least four.
	We use $\varphi_4(B_j)$ to denote one of these blocks of size at least four.
	By symmetry, we use $\psi_4(B_j)$ to denote a block of size at least four
		that contains or is adjacent to the block containing $x_j - 3t + 2$.
	In $G + \{0,1\}$, the only elements of $X$ that can be consecutive are $0$ and $1$,
		let $B_0 = \{0\}$ denote the first block of $X$.
	Thus, let $\varphi_4(B_j) = B_0$ if $x_j + 3t+1 = 0$
		and $\psi_4(B_j) = B_0$ if $x_j - 3t+2 = 0$.
\end{observation}

	We now use a two-stage discharging method to prove that there is no $r$-clique $X$ in $G$.
	In Stage 1, we assign charge to the blocks of $X$ and discharge so that all blocks have
		non-negative charge.
	In Stage 2, we assign charge to the frames of $X$ using the new charges on the blocks
		and then discharge among the frames.
	
	\begin{figure}[ht]
		\centering
		$\xymatrix{
			\text{\textbf{Stage 1:} \textit{Blocks}} & \mu(B_j) \ar[rr]^{\text{discharge}} &\ & \mu^*(B_j) \ar[d]^{\text{defines}} && \\
			\text{\textbf{Stage 2:} \textit{Frames}} & && \nu^*(F_j) \ar[rr]^{\text{discharge}} &\ & \nu'(F_j)
		}$
		\caption{The two-stage discharging method.}
	\end{figure}
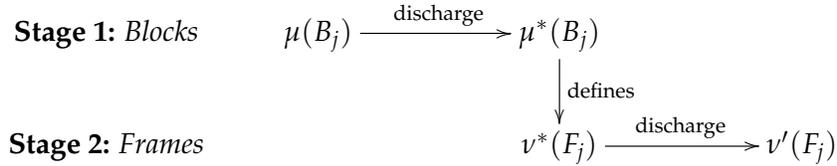
		
	We will use this framework three times, in Claims \ref{claim:threegenclique}, \ref{claim:complicateddischarging}, and \ref{claim:itIScomplicated}, 
		but we use a different set of rules for Stage 1 each time. 
	Stage 2 will always use the same discharging rule.

	\begin{claim}\label{claim:threegenclique}
		$\omega(G) < r$.
	\end{claim}

\begin{proof}[Proof of Claim \ref{claim:threegenclique}]
	Suppose $X$ is an $r$-clique in $G$.
	
	Let $\mu$ be a charge function on the blocks of $X$ defined by $\mu(B_j) = |B_j| - 3$.
	All 2-blocks have charge $-1$, 3-blocks have charge $0$, and all other blocks have positive charge.
	Moreover, the total charge on all blocks is
		\[ \sum_{j=0}^{r-1} \mu(B_j) = n - 3r = 3t - 2.\]
	We shall discharge among the blocks to form a new charge function $\mu^*$.
	
\vspace{0.5em}
\noindent\textbf{Stage 1$\alpha$:}
\renewcommand{\thestagenum}{1$\alpha$}\refstepcounter{stagenum}\label{stage:normal}
	Discharge by shifting one charge from $\varphi_2(B_j)$ to $B_j$
		for every 2-block $B_j$.
\vspace{0.5em}
		
	After Stage \ref{stage:normal},
		 $\mu^*(B_j) = 0$ when $|B_j|\in \{2,3\}$, 
		$\mu^*(B_j) = 1$ when $|B_j| = 4$, and
	\[ \mu^*(B_j) =  |B_j| - 3 - |\varphi_2^{-1}(B_j)| \geq |\varphi_2^{-1}(B_j)|\]
	when $|B_j| \geq 5$. Note that if $|\varphi_2^{-1}(B_j)| = 0$ for a block $B_j$ of size at least five, then $\mu^*(B_j) \geq 2$.
	
	Now,  $\mu^*$ is a non-negative function and $\sum_{j=0}^{r-1} \mu^*(B_j) = 3t - 2$.
	
	For every frame $F_j$, define $\nu^*(F_j)$ as $\nu^*(F_j) = \sum_{B_{j+i} \in F_j} \mu^*(B_{j+i})$.
	Since every block is contained in exactly $t$ frames, the total charge on all frames is
		\[ \sum_{j=0}^{r-1} \nu^*(F_j) = t \sum_{j=0}^{r-1} \mu^*(B_j) = t(3t-2) = r - 1.\]
	
	There must exist a frame with $\nu^*(F_j) = 0$,
		and hence contains only 2- and 3-blocks.
	If this frame contained only blocks of length three and 
		at most one block of length two, 
		then $\sigma(F_j) \in \{3t-1, 3t\}$, 
		contradicting that $X$ is a clique.
	Thus, any frame with $\nu^*(F_j) = 0$ must contain at least two 2-blocks where all blocks between are 3-blocks.
			
	For each pair $B_k, B_{k'}$ of 2-blocks that are separated only by 3-blocks,
		define $L_{k,k'}$ to be the set of frames containing both $B_k$ and $B_{k'}$,
		and $R_{k,k'}$ to be the set of frames containing both $\varphi_2(B_k)$ and $\varphi_2(B_{k'})$. 
	If $\varphi_2(B_k) = \varphi_2(B_{k'})$, then $|R_{k,k'}| = t \geq |L_{k,k'}|$.
	Otherwise, there are fewer elements between $\varphi_2(B_k)$ and $\varphi_2(B_{k'})$ than between $B_k$ and $B_{k'}$, 
		and every block between $\varphi_2(B_k)$ and $\varphi_2(B_{k'})$ has size at least three 
			(a 2-block $B_j$ between $\varphi_2(B_k)$ and $\varphi_2(B_{k'})$
				would induce a large block $\psi_2(B_j)$ between $B_k$ and $B_k'$).
	Hence, there are at least as many blocks between $B_k$ and $B_k'$ as there are 
		between $\varphi_2(B_k)$ and $\varphi_2(B_{k'})$
		and so $|L_{k,k'}| \leq |R_{k,k'}|$.
	Let $f_{k,k'} : L_{k,k'} \to R_{k,k'}$ be any injection
		where $f_{k,k'}(F_j) = F_j$ for all $F_j \in L_{k,k'} \cap R_{k,k'}$.
		
	Using these injections, we discharge among the frames to form a new charge function $\nu'$.

\vspace{0.5em}
\noindent\textbf{Stage 2:}
		For every frame $F_j$ and every pair $B_k, B_{k'}$ of 2-blocks in $F_j$ separated by only 3-blocks,
			$F_j$ pulls one charge from $f_{k,k'}(F_j)$.
\vspace{0.5em}
		
	Since every frame $F_j$ with $\nu^*(F_j) = 0$ 
		has at least one such pair $B_k, B_{k'}$ 
		and does not contain $\varphi_2(B_i)$ for any 2-block $B_i$, 
		$F_j$ pulls at least one charge but does not have any charge removed.
	Thus, $\nu'(F_j) \geq 1$.
	
	We will show that frames $F_j$ with $\nu^*(F_j) \geq 1$
		have strictly less than $\nu^*(F_j)$ charge pulled during the second stage.
	Let $\{ (B_{k_i}, B_{k_i'};F_{j_i}) : i \in \{1, \dots, \ell\} \}$ 
		be the set of pairs	$B_{k_i}, B_{k_i'}$ of 2-blocks 
		and a common frame $F_{j_i}$ 
		where $f_{k_i,k_i'}(F_{j_i}) = F_j$.
	Since each map $f_{k_i,k_i'}$ is an injection, the blocks $B_{k_i}$ are distinct for all $i \in \{1,\dots, \ell\}$, and exactly $\ell$ charge was pulled from $F_j$.
	While $B_{k_i'}$ and $B_{k_{i+1}}$ may be the same block, 
		$B_{k_1},\dots,B_{k_\ell}, B_{k_\ell'}$ are $\ell+1$ distinct 2-blocks.
	Every block $B_{k_i}$ has $\varphi_2(B_{k_i}) \in F_j$ and $\varphi_2(B_{k_\ell'}) \in F_j$.
	Thus, $\nu^*(F_j) \geq \sum_{B_i \in F_j} |\varphi_2^{-1}(B_i)| \geq \ell + 1$
		which implies $\nu'(F_j) \geq 1$.
	
	Therefore, $\nu'(F_j) \geq 1$ for all frames $F_j$, and
		$r - 1 = \sum_{j=0}^{r-1} \nu'(F_j) \geq r$, a contradiction.
	Hence, there is no clique of size $r$ in $G$, 
		proving Claim \ref{claim:threegenclique}.
	\end{proof}

	 \def\cA{{\mathcal A}}
	 \def\cB{{\mathcal B}}
	 
	 For the remaining claims, we assume $X$ is an $r$-clique in $G + \{0,1\}$
	 	where $X$ contains both $0$ and $1$.
	Then, $B_0$ is the block containing exactly $\{0\}$,
		and all other blocks from $X$ have size at least two.
	Since $0$ and $1$ are in $X$, 
		the sets $\{ 3t - 1, 3t, 3t+1\}$ and $\{ -3t-1, -3t, -3t+1\}$ 
		of consecutive elements do not intersect $X$.
	Thus, there are two blocks $B_{k_1}$ and $B_{k_2}$ so that
		$\{3t-1, 3t, 3t+1\} \subset B_{k_1}$ and $\{ -3t-1, -3t, -3t+1\} \subset B_{k_2}$.
	When $B_{k_1}$ and $B_{k_2}$ are 4-blocks, 
		then $B_0 = \psi_4(B_{k_1}) = \varphi_4(B_{k_2})$
		as in Observation \ref{obs:fourblocks}.
	 
	With the assumption that there are no 2-blocks in $X$,
		uniqueness follows through an enumerative proof similar to Claim \ref{claim:twogenunique},
		given as Claim \ref{claim:threegenunique}.
	After this claim, Claims \ref{claim:complicateddischarging} and \ref{claim:itIScomplicated}
		show that $X$ has no 2-blocks, completing the proof.
	
	\begin{claim}\label{claim:threegenunique}
		There is a unique $r$-clique in $G + \{0,1\}$ with no 2-blocks.
	\end{claim}

\begin{proof}[Proof of Claim \ref{claim:threegenunique}]
	Consider the frame family $\cF = \{ F_{jt+1} : j \in \{0,\dots, 3t - 2 \} \}$
		of $3t - 1$ disjoint frames.
	Note that the block $B_0$ is not contained in any of these frames.
	Since there are no 2-blocks, $\sigma(F_{jt+1}) \geq 3t$,
		but $\sigma(F_{jt+1}) \neq 3t$ so $\sigma(F_{jt+1}) \geq 3t+1$.
Thus, 

\[ n - 1 = \sum_{F_{jt+1} \in \cF} \sigma(F_{jt+1}) \geq (3t-1)(3t+1) = n - 3.\]

From this inequality we have $\sigma(F_{jt+1}) = 3t+1$ for all frames except
	either one frame $F_k$ with $\sigma(F_{k}) = 3t+3$ or
	two frames $F_k, F_{k'}$ with $\sigma(F_k) = \sigma(F_{k'}) = 3t+2$.

\begin{casefig}
			\scalebox{\casefigratio}{\begin{lpic}[]{"figs-unique/CaseClaim9a"(,22.5mm)}
			   \lbl[t]{50,20;$F_k$}
				\lbl[]{6,10.25;\footnotesize$x_{k}$}
				\lbl[]{103,10.25;\tiny$x_{k+t}$}
				\lbl[b]{53,12.35;\footnotesize$\leq 3t-7$ elements}
			\end{lpic}}
	\caption{\label{fig:claim:threegenunique}Claim \ref{claim:threegenunique}, $\sigma(F_k) = 3t+3$.}
\end{casefig}

	Suppose there is a frame $F_k$ with $\sigma(F_k) = 3t+3$.
Since $x_{k + t} = x_k + 3t + 3$, the elements 
\[ x_{k+t} - 3t = x_k + 3,
	\quad
	 x_{k+t} - (3t-1) = x_k + 4,\]
	\[ x_k + 3t - 1 = x_{k+t} - 4,
	 \quad\text{and}\quad
	 x_k + 3t = x_{k+t} -3,\]
	are not contained in $X$.
Since we have no 2-blocks, the elements $x_k + 2$ and $x_{k+t} - 2$ are not in $X$.
Thus, there are two blocks of size at least five in $F_k$.
This means there are $t-2$ blocks for the remaining $3t - 7$ elements, but 
	$t-2$ blocks of size at least three cover at least $3t - 6$ elements.
Hence, no frame has $\sigma(F_k)= 3t+3$.

Suppose we have exactly two frames $F_k, F_{k'} \in \cF$ 
	with $\sigma(F_k) = \sigma(F_{k'}) = 3t+2$.
If a frame $F_j$ contains a block of size at least six, then $\sigma(F_j) \geq 3t + 3$,
	so $F_k$ and $F_{k'}$ each contain either one 5-block or two 4-blocks.
However, if the first or last block (denoted by $B_j$) of $F_k$ (or $F_{k'}$) has size three,
	then $\sigma(F_k \setminus \{B_j\}) = 3t - 1$, a contradiction.
Thus, the first and last blocks of $F_k$ and $F_{k'}$ are not 3-blocks and hence
	are both 4-blocks.
Therefore, there are exactly two frames in $\cF$ containing exactly two 4-blocks
	and the rest contain exactly one 4-block, 
	for a total of $3t$ 4-blocks in $X$.

Let $\ell_1, \ell_2, \dots, \ell_{3t}$ be the indices of the 4-blocks.
Since each frame $F_{i}$ has at least one 4-block, $\ell_j \leq \ell_{j-1} + t$.
Also, if a frame $F_i$ has exactly two 4-blocks, then
	the blocks appear as the first and last blocks in $F_j$,
	giving $\ell_{j} \geq \ell_{j-1} + t - 1$.

Consider the position of $B_{\ell_1}$.
If $B_{\ell_1}$ is strictly between $B_0$ and $B_{k_1}$, then the frame $F_1$ contains two 4-blocks $B_{\ell_1}$ and $B_{k_1}$,
	and so $B_{\ell_1} = B_1$ and $B_{k_1} = B_t$.
But, there are $3t - 3$ elements between $B_0$ and $B_{k_1}$, but at least $3t-2$ elements between $B_0$ and $B_t$.
Therefore, $B_{\ell_1} = B_{k_1}$ and there are $t - 1$ 3-blocks between $B_0$ and $B_{\ell_1}$, so $\ell_1 = t - 1$.
Similarly, $B_{\ell_{3t}} = B_{k_2}$ and there are $t - 1$ 3-blocks between $B_{\ell_{3t}}$ and $B_0$, so $\ell_{3t} = (r-1) - (t-1) = 3t^2-3t+1$.

There is exactly one solution to the constraints $\ell_{j} \in \{\ell_{j-1}  + t - 1, \ell_{j-1} + t\}$ and
	$\ell_{3t} - \ell_1 = 3t^2 - 2t + 1 = (3t-1)(t-1)$ given by $\ell_j = \ell_{j-1} + t - 1$.
This uniquely describes $X$ as a clique in $G + \{0,1\}$.
\end{proof}

We now aim to show that there are no 2-blocks in an $r$-clique $X$ of $G$.
This property can be quickly checked computationally for $t \leq 4$, so we now assume that $t \geq 5$.

The problem with applying the discharging method from Claim \ref{claim:threegenclique}
	is that $B_0$ starts with charge $\mu(B_0) = -2$ and there is no clear
	place from which to pull charge to make $\mu^*(B_0)$ positive.
We define three values, $a$, $b,$ and $c$, which 
	quantify the \emph{excess charge} from Stage \ref{stage:normal} which can be redirected to $B_0$
	while still guaranteeing that all frames end with positive charge.
In Claim \ref{claim:complicateddischarging}, we assume $a + b + c \geq 3$ 
	and place all of this excess charge on $B_0$ 
	in Stage \ref{stage:complicated}, giving $\mu^*(B_0) \geq 1$;
	an identical Stage 2 discharging leads to positive charge on all frames.
In Claim \ref{claim:itIScomplicated}, Stage \ref{stage:bk1k2} pulls charge from $B_{k_1}$ and $B_{k_2}$ 
	to result in $\mu^*(B_0) = 0$ and possibly $\mu^*(B_{k_1}) = 0$ or $\mu^*(B_{k_2}) = 0$.
After Stage \ref{stage:bk1k2} and Stage 2, there may be some frames with $\nu'$-charge zero, 
	but they must contain $B_0$, $B_{k_1},$ or $B_{k_2}$.
By carefully analyzing this situation, we find a contradiction in that 
	either $X$ is not a clique or $a + b + c \geq 3$.

We now define the quantities $a$, $b$, and $c$.
	
If a block $B_j$ has size at least five and $\varphi_2^{-1}(B_j)$ is empty,
	then no charge is removed from $B_j$ in Stage \ref{stage:normal}. 
If charge is pulled from frames containing $B_j$ in Stage 2,
	there are other blocks that supply the charge required to stay positive.
Therefore, we define $a$ to be the excess $\mu$-charge
	that can be removed and maintain positive $\mu^*$-charge:
\[a = \sum_{B_j \in \cA} \left[|B_j| - 4\right],\text{ where $\cA$ is the set of blocks $B_j$ with $|B_j| \geq 5$ and $\varphi_{2}^{-1}(B_j) = \emptyset.$}\]
If a block $B_j$ has size at least five and $\varphi_2^{-1}(B_j)$ is not empty,
	charge is pulled from $B_j$ in Stage \ref{stage:normal}.
However, if $|B_{j}| > 2|\varphi_2^{-1}(B_j)| + 3$, there is more charge left
	after  Stage \ref{stage:normal} than is required in Stage 2 to maintain a positive charge on frames containing $B_j$.
We define $b$ to be the excess charge left in this situation:
\[b = \sum_{B_j \in \cB} \left[ |B_j| - (2|\varphi_2^{-1}(B_j)|+3)\right],\]
where $\cB$ is the set of blocks $B_j$ with $|B_j| \geq 5$ and	$\varphi_2^{-1}(B_j) \neq \emptyset$.

If there is a frame $F_j$ with three blocks $B_{\ell_0}, B_{\ell_1}, B_{\ell_2}$
	where $|B_{\ell_i}| \geq 4$ for all $i \in \{0,1,2\}$ and
	$\varphi_2^{-1}(B_{\ell_1}) = \emptyset$, then
	let $c = 1$; otherwise $c = 0$.
Since every frame containing $B_{\ell_1}$ also contains $B_{\ell_0}$ or $B_{\ell_2}$,
	these frames are guaranteed a positive $\nu'$-charge from $B_{\ell_0}$ or $B_{\ell_2}$,
	so the single charge on $B_{\ell_1}$ that was not pulled from previous rules is free to pass to $B_0$.

	 \begin{claim}\label{claim:complicateddischarging}
	 	Suppose $X$ is a set in $G + \{0,1\}$ with $|X| = r$.
		If $a + b + c \geq 3$, then $X$ is not a clique.
	 \end{claim}
	 
	 \begin{proof}[Proof of Claim \ref{claim:complicateddischarging}]
	 We proceed by contradiction, assuming that $a + b +  c \geq 3$ and $X$ is an $r$-clique.
	 We shall modify the two-stage discharging from Claim \ref{claim:threegenclique}
	 	with a more complicated discharging rule to handle $B_0$
		so that the result is the same contradiction: that all $r$ frames have positive
		charge, but the amount of charge over all the frames is $r - 1$.

	Let $\mu$ be the charge function on the blocks of $X$ defined by $\mu(B_j) = |B_j| - 3$.
	We discharge using Stage \ref{stage:complicated} to form the charge function $\mu^*$.

\vspace{0.5em}
\noindent\textbf{Stage 1$\beta$:} There are four discharging rules:
\renewcommand{\thestagenum}{1$\beta$}\refstepcounter{stagenum}\label{stage:complicated}	
	\begin{cem}
		\item If $|B_k| = 2$, $B_k$ pulls one charge from $\varphi_2(B_k)$. 
		
		\item $B_0$ pulls $|B_k| - 4$ charge from every block $B_k$ with $|B_k| \geq 5$ 
			and $\varphi_2^{-1}(B_k) = \emptyset$. 
		(The total charge pulled by $B_0$ in this rule is $a$.)
		
		\item $B_0$ pulls $|B_k| - (2|\varphi_2^{-1}(B_k)|+3)$ 
			charge from every block $B_k$ with $|B_k| \geq 5$ 
			and $\varphi_2^{-1}(B_k) \neq \emptyset$.
		(The total charge pulled by $B_0$ in this rule is $b$.)
			
		\item If there is a frame $F_j$ with three blocks $B_{\ell_0}, B_{\ell_1}, B_{\ell_2}$
	where $|B_{\ell_i}| \geq 4$ for all $i \in \{0,1,2\}$ and
	$\varphi_2^{-1}(B_{\ell_1}) = \emptyset$, then $B_0$ pulls one charge from $B_{\ell_1}$.
		(The amount of charge pulled by $B_0$ in this rule is $c$.)
	\end{cem}

	Since $a + b + c \geq 3$, $B_0$ pulls at least 3 charge, so $\mu^*(B_0) \geq 1$.
	Blocks of size two and three have $\mu^*$-charge zero.
	If a block $B_k$ has size four or has size at least five 
		and $\varphi_2^{-1}(B_k) = \emptyset$, then
		$\mu^*(B_k) = 1$ 
		except $B_{\ell_1}$ where $\mu^*(B_{\ell_1}) = 0$.
	 Similarly, a block $B_k$ of size at least five
	 	with $\varphi_2^{-1}(B_k) \neq \emptyset$ 
	 	has charge $\mu^*(B_k) = |\varphi_2^{-1}(B_k)|$.
	 
	For every frame $F_j$, define $\nu^*(F_j) = \sum_{B_{j+i} \in F_j} \mu^*(B_{j+i})$.
	Note that if the charge $\nu^*(F_j)$ is zero, every block in $F_j$ has zero charge
		since $\mu^*(B_{k}) \geq 0$ for all blocks.

\vspace{0.5em}
\noindent\textbf{Stage 2:}
		For every frame $F_j$ and every pair $B_k, B_{k'}$ of 2-blocks in $F_j$ separated by only 3-blocks,
			$F_j$ pulls one charge from $f_{k,k'}(F_j)$.
\vspace{0.5em}
	
	If $\nu^*(F_j) = 0$, then $F_j$ contains only blocks $B_k$ with $\mu^*(B_k) = 0$.
	These blocks are 2-blocks, 3-blocks, and $B_{\ell_1}$.
	However, any frame which contains $B_{\ell_1}$ also contains $B_{\ell_0}$ or $B_{\ell_2}$ which have positive charge.
	Thus, frames $F_j$ with $\nu^*(F_j) = 0$ contain only 2- and 3-blocks.
	Since $\sigma(F_j) \notin \{ 3t, 3t- 1\}$, $F_j$ must contain at least two 2-blocks
		$B_k, B_{k'}$,
		so $F_j$ pulls at least one charge in the second stage and loses no charge,
		so $\nu'(F_j) \geq 1$.
	
	If $\nu^*(F_j) \geq 1$, the amount of charge pulled from $F_j$ in Stage 2 is the number of
		2-block pairs $B_k, B_{k'}$ separated by 3-blocks so that $\varphi_2(B_k), \varphi_2(B_{k'}) \in F_j$.
	Observe $\mu^*(B_i) = |\varphi_2^{-1}(B_i)|$ for all blocks $B_i$ with
		$\varphi_2^{-1}(B_i) \neq \emptyset$,
		so $\nu^*(F_j) = \sum_{B_i \in F_j} \mu^*(B_i) \geq \sum_{B_i \in F_j} |\varphi_2^{-1}(B_i)|$.
	If there are $\ell$ pairs $B_{k}, B_{k'}$ that pull one charge from $F_j$ in Stage 2, 
		then there are at least $\ell+1$ 2-blocks in $\cup_{B_i \in F_j} \varphi_2^{-1}(B_i)$,
		and $\nu^*(F_j) \geq \ell + 1$.

	Therefore, $\nu'(F_j) \geq 1$ for all $j \in \{0,\dots,r-1\}$, but since 
		\[r \leq \sum_{j=0}^{r-1} \nu'(F_j) = \sum_{j=0}^{r-1} \nu^*(F_j) = t\sum_{j=0}^{r-1} \mu^*(B_j) = t\sum_{j=0}^{r-1} \mu(B_j) = t(n - 3r) = r-1,\]
		we have a contradiction, and so $X$ is not a clique.
	 \end{proof}
	 	
	\begin{claim}\label{claim:itIScomplicated}
		If $X$ is an $r$-clique in $G+ \{0,1\}$ that contains a 2-block,
			then $a + b + c\geq 3$.
	\end{claim}
	
	\begin{proof}[Proof of Claim \ref{claim:itIScomplicated}]
	We shall repeat the two-stage discharging from Claim \ref{claim:threegenclique} 
		with a simpler rule for discharging to $B_0$ than in Claim \ref{claim:complicateddischarging}.
	After this discharging is complete, we will investigate the configuration of blocks surrounding one of the 2-blocks and show that the sum $a + b + c$ has value at least three.
	
	Let $\mu$ be the charge function on the blocks of $X$ defined by $\mu(B_j) = |B_j| - 3$.
	We use Stage \ref{stage:bk1k2} to discharge among the blocks and form a charge function $\mu^*$.
	
\vspace{0.5em}
\noindent\textbf{Stage 1$\gamma$:} 
\renewcommand{\thestagenum}{1$\gamma$}\refstepcounter{stagenum}\label{stage:bk1k2}
	We have two discharging rules:
	\begin{cem}
		\item If $|B_j| = 2$, $B_j$ pulls one charge from $\varphi_2(B_j)$.
		\item $B_0$ pulls one charge from $B_{k_1}$ and one charge from $B_{k_2}$.
	\end{cem}
	
	After the first rule within Stage \ref{stage:bk1k2} 
		there is at least one charge on all
		blocks of size at least four.
	Thus, removing one more charge from each of $B_{k_1}$ and $B_{k_2}$ 
		in the second rule of Stage \ref{stage:bk1k2} maintains
		that $\mu^*(B_{k_1})$ and $\mu^*(B_{k_2})$ are non-negative.
	Since $B_0$ receives two charge and every 2-block receives one charge,
		$\mu^*(B_j)$ is non-negative after Stage \ref{stage:bk1k2} for all blocks $B_j$.
	
	Define the charge function $\nu^*(F_j) = \sum_{B_i \in F_j} \mu^*(B_i)$.

\vspace{0.5em}
\noindent\textbf{Stage 2:}
		For every frame $F_j$ and every pair $B_k, B_{k'}$ of 2-blocks in $F_j$ separated by only 3-blocks,
			$F_j$ pulls one charge from $f_{k,k'}(F_j)$.
\vspace{0.5em}
	
	Again, $\sum_{j=0}^{r-1}\nu'(F_j) = r-1$.
	Also, $\nu'(F_j) > 0$ whenever $F_j$ contains a block of order at least four that is not $B_{k_1}$ or $B_{k_2}$,
		or $F_j$ contains two 2-blocks separated only by 3-blocks.
	Since one charge was removed from $B_{k_1}$ and $B_{k_2}$ in Stage \ref{stage:bk1k2},
		the frames containing $B_{k_1}$ or $B_{k_2}$ are no longer guaranteed to have positive charge, but
		still have non-negative charge.
	In order to complete the proof of Claim \ref{claim:itIScomplicated},
		we must more closely analyze the charge function $\nu'$.

\def\cP{{\mathcal P}}

\begin{definition}[Pull sets]
	A \emph{pull set} is a set of blocks, $\cP = \{ B_{i_1},\dots, B_{i_p}\}$,
		where $|B_{i_j}| \geq 5$ 
		for all $j \in \{1,\dots, p\}$ and
		all blocks between $B_{i_j}$ and $B_{i_{j+1}}$ are 3-blocks.
	Let $\varphi_2^{-1}(\cP) = \displaystyle\cup_{B_i \in \cP} \varphi_2^{-1}(B_i)$.
	A pull set $\cP$ is \emph{perfect} if all blocks $B_i \in \cP$ have
		$|B_i| = 2|\varphi_2^{-1}(B_i)| + 3$.
	Otherwise, a pull set $\cP$ contains a block
		$B_i \in \cP$ with $|B_i| \geq 2|\varphi_2^{-1}(B_i)| + 4$ and $\cP$ is \emph{imperfect}.
	Given a pull set $\cP$, the \emph{defect} of $\cP$ is $\delta(\cP) = \sum_{B_i \in \cP} \left[\mu^*(B_i) - |\varphi_2^{-1}(B_i)|\right] - 1$.
\end{definition}

The defect $\delta(\cP)$ measures the amount of excess charge (more than one charge) 
	the pull set $\cP$ contributes to the $\nu'$-charge of any frame containing $\cP$.
Note that pull sets $\cP$ with $B_{k_1}, B_{k_2} \notin \cP$ have defect $\delta(\cP) \geq 0$, with equality if and only if $\cP$ is perfect.
Perfect pull sets $\cP$ containing $B_{k_1}$ or $B_{k_2}$ have defect $\delta(\cP) = -1$.
For a block $B_i \in \cP$, if $d \leq \mu^*(B_i) - |\varphi_2^{-1}(B_i)|$ then we say $B_i$ \textit{contributes} $d$ to the defect of $\cP$.

Consider a pull set $\cP = \{ B_{i_1},\dots, B_{i_p}\}$.
Since there are at most $3t - 4$ elements between $\varphi_2^{-1}(B_{i_p})$ and $B_{i_p}$	and all blocks from $B_{i_1}$ to $B_{i_p}$ have order at least three,
	there exists a frame that contains all blocks of $\cP$.
Therefore, every pull set is contained within \emph{some} frame.

If $B_i$ is a block with $|B_i|\geq5$, then $\cP = \{ B_i\}$ is a (not necessarily maximal) pull set, and $\{ B_i\}$ is a subset of each frame containing $B_i$.
For every frame $F_j$ and block $B_i \in F_j$ with  $|B_i|\geq5$ there is a unique maximal pull set $\cP \subseteq F_j$ containing $B_i$.
Thus, if there are multiple maximal pull sets within a frame $F_j$, then they are disjoint.

\begin{observation}\label{obs:pullsets}
	Let $X$ be an $r$-clique and $\nu'$ be the charge function on frames of $X$ after Stage \ref{stage:bk1k2} and Stage 2.
	Then, for a frame $F_j$, $\nu'(F_j)$ is at least the sum of 
		\begin{cem}
			\item the number of distinct pairs $B_{k}, B_{k'}$ of 2-blocks in $F_j$ separated only by 3-blocks,
			\item the number of 4-blocks in $F_j$ not equal to $B_{k_1}, B_{k_2}$, 
			\item $1 + \delta(\cP)$ for every maximal pull set $\cP \subseteq F_j$.
		\end{cem}
\end{observation}

In Claim \ref{claim:bstar}, we prove there exists a special block $B_*$ in a frame $F_z$ with $\nu'(F_z) = 0$.
The proof of Claim \ref{claim:bstar} reduces to three special cases which are handled in Claims \ref{claim:PullSetWithTwos}-\ref{claim:Bk2EQ4}.

Recall $\sum_{j=0}^{r-1} \nu'(F_j) = r-1$.
Let $Z$ be the number of frames $F$ with $\nu'(F) = 0$.
Then,
\begin{align*}
	\sum_{j : \nu'(F_j) > 0} \left[\nu'(F_j)-1\right] &= \sum_{j=0}^{r-1} \left[\nu'(F_j)-1\right] + Z
		= (r-1) - r + Z = Z - 1.
\end{align*}

Therefore, if there are at most $t + 1$ frames with $\nu'$-charge zero ($\nu'(F_j) = 0$),
	then the sum $\sum_{j : \nu'(F_j) > 0} [\nu'(F_j) - 1]$ is bounded above by $t$.
The proof of Claim \ref{claim:bstar} frequently reduces to a contradiction with this bound.
Claims \ref{claim:PullSetWithTwos}-\ref{claim:Bk2EQ4} provide some situations which guarantee this sum has value at least $t + 1$.

\renewcommand{\theSuperClaim}{\ref{claim:itIScomplicated}}
\begin{subclaim}\label{claim:PullSetWithTwos}
	Let $\cP$ be a pull set containing a block $B_j$.
    If $|\varphi_2^{-1}(\cP)| \geq 2$ and $x_{k_1} + 6t^2 \leq x_j \leq x_{k_2}$, 
		then there is a set $\cH$ of frames with $\sum_{F_j \in \cH} (\nu'(F_j)-1) \geq t + 1$.
\end{subclaim}

\begin{proof}[Proof of Claim \ref{claim:PullSetWithTwos}]
	Starting with $\cP^{(0)} = \cP$, we construct a sequence $\cP^{(0)}$, $\cP^{(1)}$, $\dots$, $\cP^{(\ell)}$ of pull sets with $\ell \leq \lceil\frac{t+1}{2}\rceil + 1$.
	We build $\cP^{(k)}$ by following the map $\psi_2$ 
		from $\varphi_2^{-1}(\cP^{(k-1)})$.
	This process will continue until one of the sets is not a pull set, one of the sets is an imperfect pull set, or we reach $\lceil\frac{t+1}{2}\rceil$ pull sets.
	In either case, we find a set $\cH$ of frames that satisfies the claim.

	We initialize $\cP^{(0)}$ to be $\cP$, which contains $B_j$.
	Note that it is possible that $B_j = B_{k_2}$, but otherwise $B_j$ precedes $B_{k_2}$.
	There will be at most $6t$ elements covered by the blocks starting at $\cP^{(k)}$ to the blocks preceding $\cP^{(k-1)}$.
	Note that since $x_j - x_{k_1} \geq 6t^2$, $\cP^{(k)}$ will not contain $B_{k_1}$ or $B_{k_2}$ for any $k \in \{1,\dots, \lceil \frac{t+2}{2}\rceil\}$.
			
	Let $k \geq 1$ be so that $\cP^{(k-1)}$ is a perfect pull set with $|\varphi_2^{-1}(\cP^{(k-1)})| \geq 2$.
	For every block $B_i \in \cP^{(k-1)}$, let $B_\ell$ be a 2-block in $\varphi_2^{-1}(B_i)$ and place $\psi_2(B_\ell)$ in $\cP^{(k)}$.
	Then, place any block of size at least five that is positioned between to blocks of $\cP^{(k)}$ into $\cP^{(k)}$.
	
	If $\cP^{(k)}$ is always perfect for all $k \leq \lceil\frac{t+1}{2}\rceil$, then we have pull sets $\cP^{(0)}$, $\dots$, $\cP^{(k)}$ and frames $F_{j_0}$, $F_{j_0'}$, $\dots$, $F_{j_{k-1}}$, $F_{j_{k-1}'}$, 
		where $k = \lceil\frac{t+1}{2}\rceil$.
	Thus, let $\cH = \{ F_{j_\ell}, F_{j_\ell'} : \ell \in \{1,\dots, k\}\}$ and $\sum_{F \in \cH}[\nu'(F)-1] \geq t + 1$, proving the claim.
	It remains to show that such a set $\cH$ exists if some $\cP^{(k)}$ is imperfect.
	
	If $\cP^{(k)}$ is a perfect pull set with $|\varphi_2^{-1}(\cP^{(k)})| \geq 2$, then let $F_{j_k}$ be the frame that starts at the last block of $\cP^{(k)}$ and
		$F_{j_k'}$ be the frame that ends at the first block of $\cP^{(k)}$.
	We claim that $F_{j_k}$ and $F_{j_k'}$ have $\nu'$-charge at least two.
	There are at most $3t - 4$ elements between the last block in $\cP^{(k)}$
		and the last 2-block in $\psi_2^{-1}(\cP^{(k)})$.
	If there is at most one 2-block in $F_{j_k}$, then $\sigma(F_{j_k}) \geq 2 + 3(t-2) + 5 = 3t+3$
		and $F_{j_k}$ contains all 2-blocks in $\psi_2^{-1}(\cP^{(k)})$, a contradiction.
	Therefore, the frame $F_{j_k}$ contains at least two 2-blocks.
	If those 2-blocks are separated by three blocks, they pull at least one charge in Stage 2.
	If those 2-blocks are not separated by three blocks, then either they are separated by a 4-block (which contributes at least one charge) 
		or a second maximal pull set (which contributes at least one charge).
	Thus, $\nu'(F_{j_k}) \geq 2$.
	By a symmetric argument, $F_{j_k'}$ contains two 2-blocks and has $\nu'(F_{j_k'}) \geq 2$.
	Figure \ref{fig:BuildingCPK} shows how the frames $F_{j_k}$ and $F_{j_k'}$ are placed among the pull sets $\cP^{(k-1)}$ and $\cP^{(k)}$.
			
\begin{casefig}
			\scalebox{\casefigratio}{\begin{lpic}[]{"figs-unique/ClaimPullSetChain"(,22.5mm)}
			   \lbl[t]{28,3.75;\footnotesize $\varphi_2^{-1}(\cP^{(k)})$}
			   \lbl[t]{60.4,3.75;\footnotesize $\cP^{(k)}$}
			   \lbl[t]{93.5,3.75;\footnotesize $\varphi_2^{-1}(\cP^{(k-1)})$}
			   \lbl[t]{126,3.75;\footnotesize $\cP^{(k-1)}$}
			   \lbl[t]{13,8.5; \footnotesize $F_{j_{k+1}}$}
			   \lbl[t]{42.5,8.5; \footnotesize $F_{j_k'}$}
			   \lbl[t]{78,8.5; \footnotesize $F_{j_{k}}$}
			   \lbl[t]{108.5,8.5; \footnotesize $F_{j_{k-1}'}$}
				\lbl[tr]{8,19.65;\footnotesize$\psi_2$}
				\lbl[tl]{46,19.65;\footnotesize$\varphi_2$}
				\lbl[tr]{74,19.65;\footnotesize$\psi_2$}
				\lbl[tl]{111,19.65;\footnotesize$\varphi_2$}
			\end{lpic}}
	\caption{\label{fig:BuildingCPK}Claim \ref{claim:PullSetWithTwos}, building $\cP^{(k)}$ and frames $F_{j_k}, F_{j_k'}$.} 
\end{casefig}

	If $\cP^{(k)}$ is not a perfect pull set or $|\varphi_2^{-1}(\cP^{(k)})| < 2$, 
		either $\cP^{(k)}$ is not a pull set or $\cP^{(k)}$ is an imperfect pull set.
	
	\begin{mycases}
		\case{$\cP^{(k)}$ is not a pull set.} 
			In this case, there is a non-3-block $B_j$ not in $\cP^{(k)}$ that is between two blocks $B_{\ell_1}, B_{\ell_2}$ of $\cP^{(k)}$.
			If $|B_j| \geq 5$, then $B_j$ would be added to $\cP^{(k)}$.
			Therefore, $|B_j| \in \{2,4\}$.
			
			\begin{subcases}
				\subcase{$|B_j| = 4$.}
					Every frame containing $B_j$ also contains either $B_{\ell_1}$ or $B_{\ell_2}$.
					Therefore, these $t$ frames contain a 4-block and at least one pull set with non-negative defect so they have $\nu'$-charge at least two.
					The frame starting at $B_{\ell_1}$ also contains $B_j$ and $B_{\ell_2}$,
						so this frame has two disjoint maximal pull sets and a 4-block and has $\nu'$-charge at least three.
					Therefore, if $\cH$ is the family of frames containing $B_j$, $\sum_{F \in \cH} [\nu'(F)-1] \geq t + 1$.

\begin{casefig}
			\scalebox{\casefigratio}{\begin{lpic}[]{"figs-unique/ClaimPullSetTwoBlock"(,22.5mm)}
			 	\lbl[tr]{25,19.65;\footnotesize$\psi_2$}
				\lbl[tl]{92,19.65;\footnotesize$\varphi_2$}
				\lbl[t]{15,12;$B_j$}
				\lbl[t]{41.5,12;$B_{\ell_1}$}
				\lbl[t]{75.75,12;$B_{\ell_2}$}
				\lbl[t]{97,12;$B_{g_1}$}
				\lbl[t]{132,12;$B_{g_2}$}
				\lbl[b]{57.5,14;\small $\varphi_2(B_j)$}
				\lbl[b]{114,15;\footnotesize 3-blocks}
			\end{lpic}}
	\caption{\label{fig:BuildingCPK}Claim \ref{claim:PullSetWithTwos}, Case \ref{subcase:TwoBlockInPullSet}.}
\end{casefig}
				\subcase{$|B_j| = 2$.}\label{subcase:TwoBlockInPullSet}
					Let $B_{\ell_1}$ be the last 2-block preceding $\varphi_2(B_j)$ and $B_{\ell_2}$ be the first 2-block following $\varphi_2(B_j)$.
					Note that $B_j$ is between $\psi_2(B_{\ell_1})$ and $\psi_2(B_{\ell_2})$, which must be in $\cP^{(k)}$.
					
					\subsubcaseitem Suppose $\{\varphi_2(B_j)\}$ is an imperfect pull set. 
						Then $\varphi_2(B_j)$ contributes one to the defect of any pull set containing $\varphi_2(B_j)$.
					Place all frames containing $\varphi_2(B_j)$ into $\cH$, as they have $\nu'$-charge at least two.
					Also place the frame $F$ starting at $\psi_2(B_{\ell_1})$ into $\cH$.
					If $F$ also contains $\psi_2(B_{\ell_2})$, it contains two disjoint maximal pull sets and thus has $\nu'$-charge at least two.
					Otherwise, $F$ must contain at least two 2-blocks which either pull a charge in Stage 2 or are separated by a block of size at least four 
						and $\nu'(F) \geq 2$ in any case.
					This frame family $\cH$ satisfies the claim.
					
					\subsubcaseitem Suppose $\{ \varphi_2(B_j)\}$ is a perfect pull set.
					Therefore, $|\varphi_2(B_j)| = 3 + 2h$ for some integer $h \geq 1$ and hence is odd.
					Let $B_{g_1} = \varphi_2(B_{\ell_1})$ and $B_{g_2} = \varphi_2(B_{\ell_2})$.
					Since $B_{g_1}$ and $B_{g_2}$ are in $\cP^{(k-1)}$ 
						and $\cP^{(k-1)}$ is a pull set, there are only 3-blocks between $B_{g_1}$ and $B_{g_2}$.
					Therefore, the elements $x_{g_1+1},x_{g_1+2},\dots, x_{g_2}$ have $x_{g_1 + i + 1} = x_{g_1 + i}  + 3$ for all $i \in \{1,\dots, g_2 - g_1 - 1\}$.
					The generators $3t - 1$ and $3t$ guarantee that the elements of $X$ strictly between $x_{\ell_1}$ and $x_{\ell_2}$
						are a subset of $\{ x_{\ell_1} + 2 + 3i : i \in \{0,1,\dots, g_2 - g_1\}\}$.
					Therefore, all blocks between $B_{\ell_1}$ and $B_{\ell_2}$ (including $\varphi_2(B_j)$) have size divisible by three.
					So, $|\varphi_2(B_j)|$ is an odd multiple of three, but strictly larger than three; $|\varphi_2(B_j)| \geq 9$ and $|\varphi_2^{-1}(\varphi_2(B_j))| \geq 3$.
										
					There are $t - 2$ frames containing the first three 2-blocks in $\varphi_2^{-1}(\varphi_2(B_j))$.
					Since these 2-blocks are consecutive, each frame pulls two charge in Stage 2.
					Also, let $F'$ be the frame whose last two blocks are the first two 2-blocks in $\varphi_2^{-1}(\varphi_2(B_j))$
						and let $F''$ be the frame whose first two blocks are the last two 2-blocks in $\varphi_2^{-1}(\varphi_2(B_j))$.
					Either $F'$ contains $\psi_2(B_{\ell_1})$ or contains another 2-block preceding $\varphi_2^{-1}(\varphi_2(B_j))$
						and thus $\nu'(F') \geq 2$; by symmetric argument, $\nu'(F'') \geq 2$.
					Let $\cH$ contain these frames and note that $\sum_{F \in \cH} [\nu'(F)-1] \geq t$.
					Also, add the frame $F_{i}$ whose last block is $\varphi_2(B_j)$ to $\cH$.
					If this frame is already included in $\cH$, then the charge contributed by $\varphi_2(B_j)$ was not counted in the previous bound and 
						 $\sum_{F \in \cH} [\nu'(F)-1] \geq t+1$.
					Otherwise, $F_i$ does not contain two 2-blocks from $\varphi_2^{-1}(\varphi_2(B_j))$
						and so $F_i$ spans fewer than $3t - 8$ elements preceding $\varphi_2(B_j)$.
					Thus, $F_i$ contains at least two 2-blocks which are separated either by 
						only 3-blocks (where $F_i$ pulls a charge in Stage 2) or by a block of size at least four (which contributes at least an additional charge to $F_i$)
						and so $\nu'(F_i) \geq 2$ and  $\sum_{F \in \cH} [\nu'(F)-1] \geq t+1$.
			\end{subcases}
		
		\case{$\cP^{(k)}$ is an imperfect pull set.} There is a block $B_\ell \in \cP^{(k)}$ so that $|B_{\ell}| \geq 2|\varphi_2^{-1}(B_\ell)| + 4$.
			Since $B_\ell$ contributes at least one to the defect of every pull set that contains $B_\ell$, 
				every frame containing $B_\ell$ has $\nu'$-charge at least two.
			Let $F_{j_k}$ be the frame that starts at the last block of $\cP^{(k)}$ and note that $F_{j_k}$ contains at least two 2-blocks.
			Therefore, $F_{j_k}$ either contains a pull set and two 2-blocks separated by only 3-blocks, two disjoint maximal pull sets,
				or a pull set and a 4-block and in any case has $\nu'$-charge at least two.
			If $F_{j_k}$ contains $B_\ell$, then one of the pull sets in $F_{j_k}$ is imperfect and $\nu'(F_{j_k}) \geq 3$.
			Therefore, let $\cH$ contain $F_{j_k}$ and the frames containing $B_\ell$, and $\cH$ satisfies the claim.		
	\end{mycases}
\end{proof}

\begin{subclaim}\label{claim:Bk2Big}
	Let $B_i$ be a 5-block with $x_{k_2} - 9t \leq x_i \leq x_{k_2}$.
	If every pull set $\cP$ containing $B_i$ has  $|\varphi_2^{-1}(\cP)| = |\varphi_2^{-1}(B_i)| = 1$, 
		then there is a set $\cH$ of frames with $\sum_{F_j \in \cH} (\nu'(F_j)-1) \geq t + 1$.
\end{subclaim}

\begin{proof}[Proof of Claim \ref{claim:Bk2Big}]
	Let $B_j = \psi_2(\varphi_2^{-1}(B_i))$.
	If there is a pull set $\cP$ containing $B_j$ where $|\varphi_2^{-1}(\cP)| \geq 2$, then Claim \ref{claim:PullSetWithTwos} applies to $\cP$
		and we can set $\cH$ to be the $t + 1$ frames with $\nu'$-charge at least two.
	Therefore, we assume no such pull set exists.	
	This implies $|\varphi_2^{-1}(B_j)| \in \{0,1\}$.
	
	We shall construct two disjoint sets $\cH_1$ and $\cH_2$ so that $\sum_{F \in \cH_1} [ \nu'(F_j) - 1] \geq t$ and $\sum_{F \in \cH_2} [\nu'(F) - 1] \geq 1$
		so $\cH = \cH_1 \cup \cH_2$ satisfies $\sum_{F_j \in \cH} (\nu'(F_j)-1) \geq t + 1$.
	To guarantee disjointness, there are blocks that must be contained in frames of $\cH_2$ that cannot be contained in frames of $\cH_1$.
	For instance, a frame in $\cH_2$ may contain $B_j$, but no frames in $\cH_1$ may contain $B_j$.
	
	If $\varphi_2^{-1}(B_j) = \emptyset$ or if $|B_j| \geq 6$, then $B_j$ contributes one to the defect of every pull set containing $B_j$ 
		and hence every frame containing $B_j$ has charge at least two.
	Place all of these frames in $\cH_2$ and $\sum_{F \in \cH_2} \left[\nu'(F)-1\right] \geq t$.
	
	Therefore, we may assume that $|\varphi_2^{-1}(B_j)| = 1$ and $|B_j| = 5$.
	Hence, there are exactly $3t-4$ elements between $\varphi_2^{-1}(B_j)$ and $B_j$.
	Similarly, there are exactly $3t-4$ elements between $B_j$ and $\psi_2^{-1}(B_j)$.
	In either of these regions, not all blocks may be 3-blocks.
	Let $B_{g_1}$ be the last non-3-block preceding $B_j$ and $B_{g_2}$ be the first non-3-block following $B_j$.
	We shall guarantee that all frames in $\cH_2$ contain at least one of $B_j$, $B_{g_1}$, or $B_{g_2}$.
 
	There are exactly $3t - 4$ elements between $\varphi_2^{-1}(B_i)$ and $B_i$.
	Since $3t - 4 \equiv 2 \pmod 3$, this range contains at least 
		one 2-block, two 4-blocks,
		or one block of order at least five.
	Let $B_{\ell_1}$ be the first non-3-block following $\varphi_2^{-1}(B_i)$ and $B_{\ell_2}$ be the first non-3-block preceding $B_i$.
	
	Figure \ref{fig:BuildingBk2BigCases} demonstrates the arrangement of the blocks $B_i, B_j, B_{g_1}, B_{g_2}, B_{\ell_1}$, and $B_{\ell_2}$,
		as well as two blocks $B_{h_1}$ and $B_{h_2}$ which will be selected later in a certain case based on the sizes of $B_{g_1}$ and $B_{g_2}$.
	
\begin{casefig}
			\scalebox{\casefigratio}{\begin{lpic}[]{"figs-unique/Claim14-2"(,22.5mm)}
			   \lbl[t]{130,12;\small $B_i$}
			   \lbl[t]{45,12; \small $B_j$}
			   \lbl[t]{95.5,12; \small $B_{\ell_1}$}
			   \lbl[t]{114,12; \small $B_{\ell_2}$}
			   \lbl[t]{26,12; \small $B_{g_1}$}
			   \lbl[t]{64.5,12; \small $B_{g_2}$}
			   \lbl[t]{14.5,12; \small $B_{h_1}$}
			   \lbl[t]{74,12; \small $B_{h_2}$}
			   \lbl[t]{87,4;\footnotesize $\varphi_2^{-1}(B_i)$}  
			   \lbl[t]{4,4;\footnotesize $\varphi_2^{-1}(B_j)$}  
			\end{lpic}}
	\caption{\label{fig:BuildingBk2BigCases}The blocks involved in the proof of Claim \ref{claim:Bk2Big}.} 
\end{casefig}
	
	We consider cases depending on $|B_{\ell_1}|$ and $|B_{\ell_2}|$ and either find a contradiction 
		or find at least one frame $F$ to place in $\cH_1$ so that $F$ does not contain $B_j$ or $B_{g_2}$ and $[\nu'(F)-1] \geq 1$.
	
	\begin{mycases}
		\case{$|B_{\ell_1}| = 2$.} The block $\varphi_2(B_{\ell_1})$ follows $B_{i}$.
			If all blocks between $B_i$ and $\varphi_2(B_{\ell_1})$ are 3-blocks, then $B_i$ and $\varphi_2(B_{\ell_1})$
				are contained in a common pull set $\cP$ with $|\varphi_2^{-1}(\cP)| \geq 2$, which we assumed does not happen.
			Therefore, there is a block $B_k$ between $B_i$ and $\varphi_2(B_{\ell_1})$ that is not a 3-block.
			If $B_k$ is a 2-block, then $\psi_2(B_k)$ would be a large block between $\varphi_2^{-1}(B_i)$ and $B_{\ell_1}$), a contradiction.
			If $B_k$ is a 4-block, then $\psi_4(B_k)$ would be a large block between $\varphi_2^{-1}(B_i)$ and $B_{\ell_1}$), another contradiction.
			Therefore, $|B_k| \geq 5$, but $\varphi_2^{-1}(B_k) = \emptyset$, since otherwise a 2-block from $\varphi_2^{-1}(B_k)$ would be strictly between
				$\varphi_2^{-1}(B_i)$ and $B_{\ell_1}$.
			Then, every frame containing $B_k$ has $\nu'$-charge at least two.
			The frame $F_k$ does not contain $B_j$, $B_{g_1}$, or $B_{g_2}$, so place $F_k$ in $\cH_1$.
								
		\case{$|B_{\ell_2}| \geq 5$.} If $\varphi_2^{-1}(B_{\ell_2}) \neq \emptyset$, 
				$B_{\ell_2}$ and $B_i$ are in a common pull set $\cP$ with $|\varphi_2^{-1}(\cP)| \geq 2$, but we assumed this did not happen.
			Therefore, $\varphi_2^{-1}(B_{\ell_2}) = \emptyset$ and every frame containing $B_{\ell_2}$ has $\nu'$-charge at least two.
			The frame $F_{\ell_2}$ does not contain $B_j$, $B_{g_1}$, or $B_{g_2}$, so place $F_{\ell_2}$ in $\cH_1$.
		
		\case{$|B_{\ell_1}| \geq 5$.} Since $B_{\ell_1}$ and $B_{i}$ cannot be in a pull set, there is a non-3-block between $B_{\ell_1}$ and $B_i$, 
				so $B_{\ell_1} \neq B_{\ell_2}$.
			
			\begin{subcases}	
			\subcase{$|B_{\ell_2}| = 2$.} The frame $F$ starting at $\psi_2(B_{\ell_2})$ also contains $B_{\ell_1}$ but does not contain $B_j$ or $B_{g_2}$.
			Since $\varphi_2^{-1}(B_i)$ is between $\psi_2(B_{\ell_2})$ and $B_{\ell_1}$, these blocks are in different pull sets and 
				so $\nu'(F) \geq 2$. Place $F$ in $\cH_1$.
				
			\subcase{$|B_{\ell_2}| = 4$.} The frame $F$ starting at $B_{\ell_1}$ also contains $B_{\ell_2}$ but not $B_j$ or $B_{g_2}$.
			 Since $F$ contains two 4-blocks, $\nu'(F) \geq 2$. Place $F$ in $\cH_1$.
			\end{subcases}
			
		\case{$|B_{\ell_1}| = 4$.} Since $3t - 4 \not\equiv 4 \pmod 3$, $B_{\ell_1}$ cannot be the only non-3-block between $\varphi_2^{-1}(B_i)$ and $B_i$,
			so $B_{\ell_1} \neq B_{\ell_2}$.
			Consider $F_{\ell_1}$, the frame starting at $B_{\ell_1}$.
			
			If $F_{\ell_1}$ does not contain two 2-blocks, $\sigma(F_{\ell_1}) \geq 3t - 4$ and $F_{\ell_1}$ contains $B_i$ (and $B_{\ell_2}$).
			If $|B_{\ell_2}| = 2$, then since $4 + 2 \not\equiv 3t-4 \pmod 3$ there is another block $B_k$ between $B_{\ell_1}$ and $B_i$ that 				is not a 3-block.
			Since $F_{\ell_1}$ does not contain two 2-blocks, $|B_k| \geq 4$ and therefore $\nu'(F_{\ell_1}) \geq 2$. 
			Place $F_{\ell_1}$ in $\cH_1$ and note that $F_{\ell_1}$ does not contain $B_j$, $B_{g_1}$, or $B_{g_2}$.
			
			If $F_{\ell_1}$ does contain two 2-blocks, then either those two 2-blocks pull an extra charge in Stage 2, or
				they are separated by a block of size at least four.
			In either case, $\nu'(F_{\ell_1})\geq 2$ so place $F_{\ell_1}$ in $\cH_1$.
	\end{mycases}
		
	We now turn our attention to placing frames in $\cH_2$ based on the sizes of $B_{g_1}$ and $B_{g_2}$.
	Note that $\varphi_2^{-1}(B_{g_1}) = \varphi_2^{-1}(B_{g_2}) = \emptyset$, or else Claim \ref{claim:PullSetWithTwos} applies.
	If $|B_{g_1}| \geq 5$, then every frame containing $B_{g_1}$ has $\nu'$-charge at least two, so add these $t$ frames to $\cH_2$
		to result in $\sum_{F \in \cH}[\nu'(F)-1]\geq t+1$.
	Similarly, if $|B_{g_2}| \geq 5$, then every frame containing $B_{g_2}$ has $\nu'$-charge at least two, add these frames to $\cH_2$.
	Therefore, we may assume that $|B_{g_1}|, |B_{g_2}| \in \{2, 4\}$ which provides four cases.
	
	\begin{mycases}
		\case{$|B_{g_1}| = |B_{g_2}| = 2$.}
			There are at most $3t - 4$ elements between $B_{g_1}$ and $\varphi_2(B_{g_1})$
				or between $\psi_2(B_{g_2})$ and $B_{g_2}$.
			Let $B_{h_1}$ be the last non-3-block preceding $B_{g_1}$ and $B_{h_2}$ be the first non-3-block following $B_{g_2}$.
			If $B_{h_1}$ is a 2-block,
				let $\cP_1 = \{ \varphi_2(B_{h_1}), \varphi_2(B_{g_1})\}$.
			There cannot be a 4-block $B_k$ or 2-block $B_{k'}$ between $\varphi_2(B_{h_1})$ and $\varphi_2(B_{g_1})$
				or else $\psi_4(B_k)$ or $\psi_2(B_{k'}$ would be between $B_{h_1}$ and $B_{g_1}$.
			Therefore, adding any non-3-block between $\varphi_2(B_{h_1})$ and $\varphi_2(B_{g_1})$ to $\cP_1$
				makes $\cP_1$ be a pull set where $|\varphi_2^{-1}(\cP_1)| \geq 2$ and by Claim \ref{claim:PullSetWithTwos} we are done.
			Similarly if  $B_{\ell_2}$, the first non-3-block following $B_{g_2}$, is a 2-block, then
				let $\cP_2 = \{ \varphi_2(B_{\ell_2}), \varphi_2(B_{g_2})\}$ and we can expand $\cP_2$ to a pull set
				where $|\varphi_2^{-1}(\cP_2)| \geq 2$ and by Claim \ref{claim:PullSetWithTwos} we are done.
			Since we assumed this is not the case, $B_{h_1}$ and $B_{h_2}$ have
				size at least four.
			Either $\psi_2(B_{g_2}) = B_{h_1}$ or $B_{h_1}$ follows $\psi_2(B_{g_2})$.
			Either $\varphi_2(B_{g_1}) = B_{h_2}$ or $B_{h_2}$ precedes $\psi_2(B_{g_2})$.
			Thus, every frame containing $B_j$ also contains $B_{h_1}$ or $B_{h_2}$
				and thus contains at least a pull set and a 4-block or two maximal pull sets which implies the frame has $\nu'$-charge at least two.
			Place these frames in $\cH_2$.
			
		\case{$|B_{g_1}| = |B_{g_2}| = 4$.}
			There are at most $3t - 3$ elements between  $B_{g_1}$ and $\varphi_4(B_{g_1})$
				or between $\psi_4(B_{g_2})$ and $B_{g_2}$.
			Since $B_{g_1}$ is the last non-3-block preceding $B_j$,
				either $\psi_4(B_{g_2}) = B_{g_1}$ or $\psi_4(B_{g_2})$ precedes $B_{g_1}$.
			Similarly, either $\varphi_4(B_{g_1}) = B_{g_2}$ or $\varphi_4(B_{g_1})$ follows $B_{g_1}$.
			Therefore, every frame containing $B_j$ also contains $B_{g_1}$ or $B_{g_2}$
				and thus contains a pull set and a 4-block which implies the  frame has $\nu'$-charge at least two.
			Place these frames in $\cH_2$.
		
		\case{$|B_{g_1}| = 2$ and $|B_{g_2}| = 4$.}\label{case:Bk2BigG2G2EQ4}
			There are at most $3t - 4$ elements between $B_{g_1}$ and $\varphi_2(B_{g_1})$
				and at most $3t - 3$ elements between $\psi_4(B_{g_2})$ and $B_{g_2}$.
			Let $B_{h_1}$ be the last non-3-block preceding $B_{g_1}$.
			If $B_{h_1}$ a 2-block,
				then there is a pull set $\cP_1 = \{ \varphi_2(B_{h_1}), \varphi_2(B_{g_1})\}$ 
				where $|\varphi_2^{-1}(\cP_1)| \geq 2$.
			We assumed this is not the case, so $|B_{h_1}| \geq 4$.
			Either $B_{h_1} = \psi_4(B_{g_2})$ or $B_{h_1}$ follows $\psi_4(B_{g_2})$.
			Therefore, every frame containing $B_j$ also contains $B_{h_1}$ or $B_{g_2}$
				and thus contains a pull set and a 4-block or two maximal pull sets
				which implies the  frame has $\nu'$-charge at least two.
			Place these frames in $\cH_2$.

		\case{$|B_{g_1}| = 4$ and $|B_{g_2}| = 2$.} This case is symmetric to Case \ref{case:Bk2BigG2G2EQ4}.
	\end{mycases}
	
	Thus, $\cH = \cH_1 \cup \cH_2$ has been selected from $\cH_1$ and $\cH_2$
		so that $\sum_{F \in \cH} [\nu'(F)-1] \geq t + 1$.
\end{proof}

\begin{subclaim}\label{claim:Bk2EQ4}
	 If there is a block $B_\ell$ with $|B_{\ell}| = 4$, $x_{k_2} - 12t \leq x_\ell \leq x_{k_2}$, and there is a block $B_i$ 
	 		between $\psi_4(B_{\ell})$ and $B_{\ell}$ with $|B_i| \neq 3$, then 
			there is a set $\cH$ of frames so that $\displaystyle\sum_{F \in \cH} \left[\nu'(F) - 1\right] \geq t + 1$.
\end{subclaim}

\begin{proof}[Proof of Claim \ref{claim:Bk2EQ4}]
	Note that it may be the case that $B_{\ell} = B_{k_2}$.
	For the remainder of the proof, $B_\ell$ will not be used to bound the $\nu'$-charge of frames in $\cH$ 
		and all other blocks will contain elements between $x_\ell - 12t$ and $x_\ell$, 
		so these blocks will not be one of $B_0$, $B_{k_1}$, or $B_{k_2}$.

	Let $\psi_4^{(d)}$ denote the $d$th composition of the map $\psi_4$.
	Let $D\geq 1$ be the first integer so that $|\psi_4^{(D)}(B_\ell)|\neq 4$, if it exists.
	We will select blocks $B_{\ell_1}, B_{\ell_2}, B_{\ell_3}$, and $B_{\ell_4}$ based on the value of $D$.
	For all $d \leq D$, let $B_{\ell_d} = \psi_4^{(d)}(B_\ell)$.
	
	If $D < 4$, then we must use different methods to find the remaining blocks $B_{\ell_d}$.
	Note that $|B_{\ell_D}| \geq 5$.
	If $|\varphi_2^{-1}(B_{\ell_D})| \geq 2$, then by Claim \ref{claim:PullSetWithTwos} we are done.
	If $|\varphi_2^{-1}(B_{\ell_D})| = 1$ and $|B_{\ell_D}| = 5$, then 
		either there is a pull set $\cP$ containing $B_{\ell_D}$ with $|\varphi_2^{-1}(\cP)| \geq 2$ and by  Claim \ref{claim:PullSetWithTwos} we are done
		or every pull set $\cP$ containing $B_{\ell_D}$ has $|\varphi_2^{-1}(\cP)| = 1$ and by Claim \ref{claim:Bk2Big} we are done.
	Therefore, there are two remaining cases for $B_{\ell_D}$:
		either (a) $\varphi_2^{-1}(B_{\ell_D}) = \emptyset$,
		or (b) $|\varphi_2^{-1}(B_{\ell_D})| = 1$ and $|B_{\ell_D}| \geq 6$.

	We consider cases based on $|B_i|$.
	
\begin{casefig}
			\scalebox{\casefigratio}{\begin{lpic}[]{"figs-unique/ClaimFourBlockChain"(,22.5mm)}
			   \lbl[t]{73.5,12;\footnotesize $B_{\ell}$}
			   \lbl[t]{25.25,12;\footnotesize $B_{\ell_1}$}
			   \lbl[t]{58,12; \footnotesize $B_{i}$}
			   \lbl[t]{42.25,12;\footnotesize $B_{i_1}$}
			   	\lbl[t]{44,20.5;\footnotesize$\psi_2$}
				\lbl[b]{60,0.75;\footnotesize$\psi_4$}
			\end{lpic}}
	\caption{\label{fig:FourBlockChain}Claim \ref{claim:Bk2EQ4}, Case \ref{case:Bk2EQ4Bi2}: $|B_{\ell}| = 4$ and $|B_i| = 2$, shown with $D \geq 4$.}
\end{casefig}
	\begin{mycases}
		\case{$|B_i| = 2$.}\label{case:Bk2EQ4Bi2}
			Let $B_{i_1} = \psi_2(B_i)$.
			$B_{i_1}$ is a block of size at least five preceding $B_{\ell_1}$.
			If there exists a pull set $\cP$ containing $B_{i_1}$ so that  $|\varphi_2^{-1}(\cP)| \geq 2$, then by Claim \ref{claim:PullSetWithTwos} we are done.
			Therefore, $|\varphi_2^{-1}(B_{i_1})| \in \{0,1\}$.
			
			\begin{subcases}
				\subcase{\label{ssc:Bk2EQ4Bi2nonempty} Suppose $|\varphi_2^{-1}(B_{i_1})| = 1$.} 
				If $|B_{i_1}| = 5$, then by Claim \ref{claim:Bk2Big} we are done.			
				Therefore, $|B_{i_1}| \geq 6$ and $B_{i_1}$ contributes at least one to the defect of every pull set containing $B_{i_1}$,
					so every frame containing $B_{i_1}$ has $\nu'$-charge at least two.
				Place these frames in $\cH$.
				
				There are at most $3t - 4$ elements between $B_{i_1}$ and $B_i$, so if
					does not contain $B_{\ell_1}$, then $F_{i_1}$ contains at least two 2-blocks.
				If these 2-blocks are separated only by 3-blocks, then $\nu'(F_{i_1}) \geq 3$
					because the imperfect pull set containing $B_{i_1}$ 
					contributes two charge and these 2-blocks pull one charge in Stage 2. 
				Otherwise, these 2-blocks are separated by some block of order at least four.
				Therefore, $\nu'(F_{i_1}) \geq 3$ since the imperfect pull set containing $B_{i_1}$ 
					contributes two charge and either the 4-blocks between the 2-blocks contributes one charge
					or the block of size at least five between the 2-blocks is contained in a pull set 
					that contributes at least one charge.
				Thus, if $F_{i_1}$ does not contain $B_{\ell_1}$, we are done.
				We now assume that $B_{\ell_1} \in F_{i_1}$.
				
				If $D \geq 2$, then $|B_{\ell_1}| = 4$.
					Then $\nu'(F_{i_1}) \geq 3$ because the imperfect pull set containing $B_{i_1}$ 
					contributes two charge and $B_{\ell_1}$ contributes one charge.

				If $D = 1$, then $|B_{\ell_1}| \geq 5$.
					If $\varphi_2^{-1}(B_{\ell_1}) = \emptyset$, then
						$B_{\ell_1}$ contributes two charge to $F_{i_1}$ and 
						$\nu'(F_{i_1}) \geq 4$.
					Otherwise $|\varphi_2^{-1}(B_{\ell_1})| = 1$ and $|B_{\ell_1}| \geq 6$, 
						so $B_{\ell_1}$ contributes at least one to the defect of any pull set containing $B_{\ell_1}$
						and thus $\nu'(F_{i_1}) \geq 3$.
			
				Since $\cH$ contains $t$ frames of $\nu'$-charge at least two and at least one frame ($F_{i_1}$)
					with $\nu'$-charge at least three, 
					$\displaystyle\sum_{F \in \cH} \left[\nu'(F) - 1\right] \geq t + 1$.
			
				\subcase{Suppose $|\varphi_2^{-1}(B_{i_1})| = 0$.}
				$B_{i_1}$ contributes at least two to the $\nu'$-charge for every frame containing $B_{i_1}$.
				Place these $t$ frames in $\cH$.
				As in Case \ref{ssc:Bk2EQ4Bi2nonempty}, the frame $F_{i_1}$ must have charge $\nu'(F_{i_1}) \geq 3$ and
					$\displaystyle\sum_{F \in \cH} \left[\nu'(F) - 1\right] \geq t + 1$.
					
		\end{subcases}

\begin{casefig}
			\scalebox{\casefigratio}{\begin{lpic}[]{"figs-unique/ClaimFourBlockChainC"(,22.5mm)}
			   \lbl[t]{41,12;\footnotesize $B_{\ell}$}
			   \lbl[t]{8.5,12;\footnotesize $B_{\ell_1}$}
			   \lbl[t]{23.5,12; \footnotesize $B_{i}$}
				\lbl[b]{27,0.75;\footnotesize$\psi_4$}
			\end{lpic}}
	\caption{\label{fig:FourBlockChain}Claim \ref{claim:Bk2EQ4}, Case \ref{case:Bk2EQ4Bi5}: $|B_{\ell}| = 4$ and $|B_i| = 2$, shown with $D \geq 4$.}
\end{casefig}
		\case{$|B_i| \geq 5$.}\label{case:Bk2EQ4Bi5}
			Let $\cH$ be the frames containing $B_i$.
			If there exists a pull set $\cP$ containing $B_i$ with $|\varphi_2^{-1}(\cP)| \geq 2$, then by Claim \ref{claim:PullSetWithTwos}, we are done.
			If $|B_i| = 5$ and $|\varphi_2^{-1}(B_i)| = 1$, then by Claim \ref{claim:Bk2Big}, we are done.
			Therefore, either $\varphi_2^{-1}(B_i) = \emptyset$ and $|B_i| \geq 5$, or
				$|\varphi_2^{-1}(B_i)| = 1$ and $|B_i| \geq 6$.
			In either case, $B_i$ contributes at least two charge to every frame in $\cH$.
			
			Consider the frame $F_{i-t+1} \in \cH$ where $B_i$ is the last block of $F_{i-t+1}$.
			
			If $F_{i-t+1}$ has fewer than two 2-blocks, then $\sigma(F_{i-t+1}) \geq 2 + 3(t-2) + |B_i| \geq 3t+1$.
			Since there are at most $3t - 3$ elements between $B_{\ell_1}$ and $B_\ell$,
				then $B_{\ell_1} \in F_{i-t+1}$ when $F_{i-t+1}$ has fewer than two 2-blocks.
			If $|B_{\ell_1}| = 4$, then $B_{\ell_1}$ contributes another charge to $F_{i-t+1}$ and $\nu'(F_{i-t+1}) \geq 3$.
			If $|B_{\ell_1}| \geq 5$ and $\varphi_2^{-1}(B_{\ell_1}) = \emptyset$ and $B_{\ell_1}$ contributes at least two charge to $F_{i-t+1}$ and $\nu'(F_{i-t+1}) \geq 4$.
			Otherwise, $|B_{\ell_1}| \geq 5$ and $\varphi_2^{-1}(B_{\ell_1}) \neq \emptyset$.
			Since $B_i$ is not contained within any pull set $\cP$ with $|\varphi_2^{-1}(\cP)| \geq 2$,
				then either $\varphi_2^{-1}(B_i) = \emptyset$ or $B_i$ and $B_{\ell_1}$ are not contained in a common pull set.
			In either case, $B_{\ell_1}$ contributes at least one more charge to $F_{i-t+1}$ and $\nu'(F_{i-t+1}) \geq 3$.
			
			If $F_{i-t+1}$ has two or more 2-blocks, then either two 2-blocks are separated only by 3-blocks and contribute an extra charge to $F_{i-t+1}$
				or they are separated by a block of size at least four which is not in a pull set with $B_i$ and contributes an extra charge to $F_{i-t+1}$.
						
			Therefore, $\nu'(F_{i-t+1}) \geq 3$ and $\sum_{F \in \cH} \left[\nu'(F)-1\right] \geq t + 1$.

			\begin{casefig}
			\scalebox{\casefigratio}{\begin{lpic}[]{"figs-unique/ClaimFourBlockChainB"(,22.5mm)}
			   \lbl[t]{137,12;\footnotesize $B_{\ell}$}
			   \lbl[t]{105,12;\footnotesize $B_{\ell_1}$}
			   \lbl[t]{70.5,12;\footnotesize $B_{\ell_2}$}
			   \lbl[t]{36.5,12;\footnotesize $B_{\ell_3}$}
			   \lbl[t]{88.5,12;\footnotesize $B_{i_1}$}
			   \lbl[t]{55,12;\footnotesize $B_{i_2}$}
			   \lbl[t]{22.75,12;\footnotesize $B_{i_3}$}
			   \lbl[t]{5,12;\footnotesize $B_{\ell_4}$}
			   \lbl[t]{120.75,12;\footnotesize $B_{i}$}
				\lbl[t]{42,20.5;\footnotesize$\psi_4$}
				\lbl[t]{73,20.5;\footnotesize$\psi_4$}
				\lbl[t]{108,20.5;\footnotesize$\psi_4$}
				\lbl[b]{124.5,0.75;\footnotesize$\psi_4$}
				\lbl[b]{90,0.75;\footnotesize$\psi_4$}
				\lbl[b]{55,0.75;\footnotesize$\psi_4$}
				\lbl[b]{21,0.75;\footnotesize$\psi_4$}
			\end{lpic}}
	\caption{\label{fig:FourBlockChainB}Claim \ref{claim:Bk2EQ4},  Case \ref{case:Bk2EQ4Bi4}: $|B_{\ell}| = 4$ and $|B_i| = 4$, shown with $D \geq 4, D' \geq 3$.}
\end{casefig}
		\case{$|B_i| = 4$.}\label{case:Bk2EQ4Bi4}
			Let $D'\geq 1$ be the first integer so that $|\psi_4^{(D')}(B_i)| \neq 4$.
			For $d \in \{1,\dots,D'\}$, define $B_{i_d} = \psi_4^{(d)}(B_i)$.
			
			\begin{subcases}
				\subcase{$D \geq 4$ and $D' \geq 3$.}
				Note that for $j \in \{1,2,3\}$, $B_{i_j}$ is between $B_{\ell_{j+1}}$ and $B_{\ell_j}$.
				There are at most $3t - 3$ elements between $B_{\ell_{j+1}}$ and $B_{\ell_j}$, 
					so every frame $F$ containing $B_{i_j}$ either contains one of $B_{\ell_{j+1}}$ or $B_{\ell_j}$
					or has $\sigma(F) \leq 3t - 4$.
				If $F$ contains $B_{i_j}$ and one of $B_{\ell_{j+1}}$ or $B_{\ell_j}$,
					then either $\nu'(F) \geq 2$ or $B_{i_j}$ is contained in a perfect pull set $\cP$ with the other block
					and $|\varphi_2^{-1}(\cP)| \geq 2$ so by Claim \ref{claim:PullSetWithTwos} we are done.
				If $\sigma(F) \leq 3t - 3$, then there are at least three 2-blocks in $F$.
				At least two of these 2-blocks are on a common side of $B_{i_j}$, and either they are separated only by 3-blocks (and pull an extra charge to $F$)
					or they are separated by a block of size at least four (which contributes an extra charge to $F$).
				Therefore, every frame containing $B_{i_j}$ has $\nu'$-charge at least two.
				Build $\cH$ from the frames containing $B_{i_1}$ and the frames containing $B_{i_3}$.
				Then $\sum_{F\in \cH}\left[\nu'(F)-1\right] \geq 2t$.
				
				\subcase{$D' < D < 4$.}
					By definition, $|B_{i_{D'}}| \geq 5$.					
					Let $\cH$ be the set of frames containing $B_{i_{D'}}$.
					
					If there exists a pull set $\cP$ containing $B_{i_{D'}}$ so that $|\varphi_2^{-1}(\cP)| \geq 2$ then by Claim \ref{claim:PullSetWithTwos} we are done.
					If $|\varphi_2^{-1}(B_{i_{D'}})| = 1$ and $|B_{i_{D'}}| = 5$, then by Claim \ref{claim:Bk2Big} we are done.
					Therefore, $B_{i_{D'}}$ contributes at least one to the defect of every pull set containing $B_{i_{D'}}$ 
						and hence every frame containing $B_{i_{D'}}$ has $\nu'$-charge at least two.
					
					The block $B_{i_{D'}}$ is between $B_{\ell_{D'+1}}$ and $B_{\ell_{D'}}$ 
						and there are at most $3t - 3$ elements between $B_{\ell_{D'+1}}$ and $B_{\ell_{D'}}$.
					Consider the frame $F_{i_{D'}}$, which has $B_{i_{D'}}$ as the first block.
					If $F_{i_{D'}}$ contains $B_{\ell_{D'}}$, then $\nu'(F_{i_{D'}}) \geq 3$ since $B_{\ell_{D'}}$ is a 4-block 
						and $B_{i_{D'}}$ contributed two charge to $F_{i_{D'}}$.
					Otherwise, $\sigma(F_{i_{D'}}) \leq 3t - 3$ and $F_{i_{D'}}$ contains at least two 2-blocks.
					Either these 2-blocks are separated by 3-blocks and pull a charge in Stage 2,
						or there is a block of size at least four between these blocks 
						and contributes at least one more charge to $F_{i_{D'}}$.
					Therefore, $\nu'(F_{i_{D'}}) \geq 3$ and $\sum_{F \in \cH}\left[\nu'(F)-1\right] \geq t + 1$.

				\subcase{$D \leq D' < 4$.}
					By definition, $|B_{\ell_D}| \geq 5$.	
					Let $\cH$ be the set of frames containing $B_{\ell_D}$.
					
					If there exists a pull set $\cP$ containing $B_{\ell_D}$ so that $|\varphi_2^{-1}(\cP)| \geq 2$ then by Claim \ref{claim:PullSetWithTwos} we are done.
					If $|\varphi_2^{-1}(B_{\ell_D})| = 1$ and $|B_{\ell_D}| = 5$, then by Claim \ref{claim:Bk2Big} we are done.
					Therefore, $B_{\ell_D}$ contributes at least one to the defect of every pull set containing $B_{\ell_D}$ 
						and hence every frame containing $B_{\ell_D}$ has $\nu'$-charge at least two.
					
					The block $B_{\ell_D}$ is between $B_{i_D}$ and $B_{i_{D-1}}$ and there are at most $3t - 3$ elements between $B_{i_D}$ and $B_{i_{D-1}}$.
					Consider the frame $F_{\ell_D}$, which has $B_{\ell_D}$ as the first block.
					If $F_{\ell_D}$ contains $B_{i_{D-1}}$, then $\nu'(F_{\ell_D}) \geq 3$ since $B_{i_{D-1}}$ is a 4-block 
						and $B_{\ell_D}$ contributed two charge.
					Otherwise, $\sigma(F_{\ell_D}) \leq 3t - 3$ and $F_{\ell_D}$ contains at least two 2-blocks.
					Either these 2-blocks are separated by 3-blocks and pull a charge in Stage 2,
						or there is a block of size at least four between these blocks 
						and contributes at least one more charge to $F_{\ell_D}$.
					Therefore, $\nu'(F_{\ell_D}) \geq 3$ and $\sum_{F \in \cH}\left[\nu'(F)-1\right] \geq t + 1$.\qedhere
			\end{subcases}
	\end{mycases}
\end{proof}

	Since $\sum_{j=1}^r \nu'(F_j) = r - 1$, there is some frame $F_{z}$ with $\nu'(F_{z}) = 0$.
	Also, the only frames where $\nu'(F_j)$ may be zero are those containing $B_0$, $B_{k_1}$, or $B_{k_2}$.
	
\begin{subclaim}\label{claim:bstar}
	There exists a block $B_*$ and a frame $F_z$ so that $B_* \in F_z$,
		$\nu'(F_z) = 0$,
		and for all 2-blocks $B_j$, $B_*$ does not appear between $\psi_2(B_j)$ and $\varphi_2(B_j)$, inclusive.
\end{subclaim}

\begin{proof}[Proof of Claim \ref{claim:bstar}]
	Using any frame $F_{z}$ with $\nu'(F_{z}) = 0$, 
		we will show that there is a block $B_* \in \{ B_0, B_{k_1}, B_{k_2} \} \cap F_z$ 
		so that for all 2-blocks $B_j$,
		$B_*$ does not appear between $\psi_2(B_j)$ and $\varphi_2(B_j)$.

Consider five cases based on which blocks ($B_0$, $B_{k_1},$ or $B_{k_2}$) are within $F_z$ and
	if there are other frames with zero charge.
		
\begin{mycases}
	\case{For some $i \in \{1,2\}$, $B_{k_i} \in F_{z}$ and $|B_{k_i}| = 4$.}\label{case:Bk1Small}
	Since $\nu'(F_z) = 0$, we must have that either $\nu^*(F_z) = 0$ or $\nu^*(F_z) > 0$ and charge was pulled from $F_z$ in Stage 2.

If $\nu^*(F_z) = 0$, then $F_z$ contains no block of size at least four other than $B_{k_i}$.
If there are no 2-blocks, then every block of $F_z \setminus \{B_{k_i}\}$ is a 3-block and $\sigma(F_z) = 3t + 1$. 
All 2-blocks $B_j$ have at most $3t-4$ elements between $B_j$ and $\varphi_2(B_j)$ or between $\psi_2(B_j)$ and $B_j$, so there are not enough elements to fit $F_z$ in these ranges and hence $B_* = B_{k_i}$ suffices.

If there is exactly one 2-block in $F_z$, then $\sigma(F_z) = 3t$, a contradiction.
Similarly, if there are exactly two 2-blocks in $F_z$, then $\sigma(F_z) = 3t-1$, a contradiction.
Hence, there are at least three 2-blocks in $F_z$ and some pair of 2-blocks is separated by only 3-blocks,
	so Stage 2 pulled at least one charge from another frame, contradicting $\nu'(F_z) = 0$.
	
If $\nu^*(F_z) > 0$, then there must be at least one block of order four or more other than $B_{k_i}$.
If any of these blocks are 4-blocks, then the positive charge contributed cannot be removed by Stage 2.
If any of these blocks have size at least five, the associated maximal pull set in $F_z$ does not contain $B_{k_1}$ or $B_{k_2}$ so
	the defect is non-negative and Stage 2 leaves at least one charge, so $\nu'(F_z) > 0$.
	
	\case{$B_{k_1} \in F_{z}$ and $|B_{k_1}| \geq 5$.}\label{case:Bk1Big}
	Since $x_0 + 3t \in B_{k_1}$ and $B_0$ is not included in $\varphi_2^{-1}(B_{k_1})$,
		we have $|B_{k_1}| \geq 2|\varphi_2^{-1}(B_{k_1})| + 4$.
	Thus the maximal pull set in $F_z$ containing $B_{k_1}$ is imperfect and $\nu'(F_z) > 0$, a contradiction.

	\case{$B_0 \in F_z$, there are no 2-blocks in $F_z$, and $F_z$ does not contain $B_{k_1}$ or $B_{k_2}$.}
		\label{case:B0allthreeblocks}
		Since $\nu'(F_z) = 0$, there is no block in $F_z$ with size at least four,
			hence $F_z$ contains $t-1$ 3-blocks and $B_0$, so $\sigma(F_z) = 3t - 2$.
		For a 2-block $B_j$, there are at most $3t-4$ elements contained in the blocks strictly 
			between $B_j$ and $\varphi_2(B_j)$ 
			or the blocks strictly between $B_j$ and $\psi_2(B_j)$.
		Then, if $B_0$ appears between $\psi_2(B_j)$ and $\varphi_2(B_j)$, then 
			one of $\psi_2(B_j)$, $B_j$, or $\varphi_2(B_j)$ must be within $F_z$, a contradiction.
		Thus, $B_* = B_j$ suffices.
	
	\case{$B_0 \in F_z$, $F_z$ contains at least one 2-block, $F_z$ does not contain $B_{k_1}$ or $B_{k_2}$.}\label{case:B0withTwoBlock}
	Since $F_z$ does not contain $B_{k_1}$ or $B_{k_2}$, any block of size at least four implies $\nu'(F_z) \geq 1$, a contradiction.
	Further, if there are at least three 2-blocks in $F_z$, then two 2-blocks are separated by only 3-blocks and $F_z$ pulls a charge in Stage 2, a contradiction.
	Therefore, $F_z$ contains either one or two 2-blocks.
	If there are two 2-blocks, there must be one 2-block (call it $B_{i_1}$) preceding $B_0$ and another (call it $B_{i_2}$)  following $B_0$.
	In either case, $\sigma(F_z) \in \{3t - 4, 3t - 3\}$.
	
	Let $B_{\ell_1}$ be the block immediately following $F_z$ and
		$B_{\ell_2}$ be the block immediately preceding $F_z$.
	If $\sigma(F_z) = 3t - 3$ and $B_{\ell_j}$ has size two or three (for some $j \in \{1,2\}$), 
		then $\sigma(F_z \cup \{B_{\ell_j}\}) \in \{3t-1,3t\}$, a contradiction.
	If $\sigma(F_z) = 3t - 4$ and $|B_{\ell_j}| \in \{3, 4\}$ (for some $j \in \{1,2\}$),
		then $\sigma(F_z \cup \{B_{\ell_i}\}) \in \{3t - 1,3t\}$, a contradiction. 
	Hence, $|B_{\ell_1}|, |B_{\ell_2}| \geq 4$ when exactly one 2-block exists, 
		or $|B_{\ell_j}| = 2$ and the 2-block $B_{i_j}$ is between $B_0$ and $B_{\ell_j}$ (and every frame containing both $B_{i_j}$ and $B_{\ell_j}$ pulls a charge in Stage 2).
	Since all other frames containing $B_0$ contain either $B_{\ell_1}$ or $B_{\ell_2}$, they have positive $\nu'$-charge.
	Therefore, $F_z$ is the \emph{only} frame with zero charge and $\displaystyle\sum_{j: \nu'(F_j) > 0} [\nu'(F_j) - 1] = 0$.
	Hence, if there exists any frame with $\nu'$-charge at least two, we have a contradiction.
	
	We consider if $B_{i_1}$ and $B_{i_2}$ both exist and whether or not $\psi_2(B_{i_j})$ is equal to $B_{k_2}$ for some $j$.
	
	\begin{subcases}
		\subcase{$\psi_2(B_{i_j}) = B_{k_2}$ for some $j \in \{1,2\}$.}
			Since $|B_{k_2}| \geq 2|\psi_2^{-1}(B_{k_2})| + 4$, $|B_{k_2}| \geq 6$.
			If $\varphi_2^{-1}(B_{k_2}) = \emptyset$, then $\mu^*(B_{k_2}) \geq 2$ and
				every frame containing $B_{k_2}$ has $\nu'$-charge at least two, a contradiction.
			If $|\varphi_2^{-1}(B_{k_2})| \geq 2$, Claim \ref{claim:PullSetWithTwos} implies 
				$\displaystyle\sum_{j : \nu'(F_j)>0} [\nu'(F_j) - 1] \geq t+1$, a contradiction.
			Thus, $|\varphi_2^{-1}(B_{k_2})| = 1$.
			Let $B_g$ be the unique 2-block in $\varphi_2^{-1}(B_{k_2})$.
			Note that $|\psi_2(B_g)| \geq 5$.
			If $|\psi_2(B_g)| \geq 2|\varphi_2^{-1}(\psi_2(B_g))| + 4$, then
				$\psi_2(B_g)$ contributes one to the defect of every pull set containing $\psi_2(B_g)$ 
				and every frame containing $\psi_2(B_g)$ has $\nu'$-charge at least two, a contradiction.
			Thus, $|\psi_2(B_g)| = 2|\varphi_2^{-1}(\psi_2(B_g))| + 3 \geq 5$
				and	every pull set $\cP$ which contains $\psi_2(B_g)$ has $|\varphi_2^{-1}(\cP)| \geq 1$.
			If any such pull set has $|\varphi_2^{-1}(\cP)| \geq 2$, then Claim \ref{claim:PullSetWithTwos} implies
					$\sum_{j : \nu'(F_j) > 0} [\nu'(F_j) - 1] \geq t+1$.
			Otherwise, every pull set containing $\psi_2(B_g)$ has $|\varphi_2^{-1}(\cP)| = 1$
					and Claim \ref{claim:Bk2Big} implies $\sum_{j : \nu'(F_j) > 0} [\nu'(F_j) - 1] \geq t+1$.
		
		\subcase{$\psi_2(B_{i_j}) \neq B_{k_2}$ for both $j \in \{1,2\}$.}
			Consider some $j \in \{1,2\}$ so that $B_{i_j}$ exists.
			If $|\varphi_2^{-1}(\psi_2(B_{i_j}))| \geq 2$, then Claim \ref{claim:PullSetWithTwos} provides a contradiction.
			If $|\psi_2(B_{i_j})| \geq 2|\varphi_2^{-1}(\psi_2(B_{i_j}))| + 4$, then $\psi_2(B_{i_j})$ contributes
				at least one to the defect of any pull set containing $\psi_2(B_{i_j})$,
				and every frame containing $\psi_2(B_{i_j})$ has $\nu'$-charge at least two, a contradiction.
			Therefore, the size of $\varphi_2^{-1}(\psi_2(B_{i_j}))$ is 1 and $|\psi_2(B_{i_j})| = 5$.

			Every pull set $\cP$ which contains $\psi_2(B_{i_j})$ has $|\varphi_2^{-1}(\cP)| \geq 1$.
			If any such pull set has $|\varphi_2^{-1}(\cP)| \geq 2$, then Claim \ref{claim:PullSetWithTwos} provides a contradiction.
			Otherwise, every pull set containing $\psi_2(B_{i_j})$ has $|\varphi_2^{-1}(\cP)| = 1$
					and Claim \ref{claim:Bk2Big} provides a contradiction.
	\end{subcases}

	\case{$B_{k_2} \in F_{z}$ and $|B_{k_2}| \geq 5$.}
	If $|B_{k_2}| \geq 2|\varphi_2^{-1}(B_{k_2})| + 4$,
		then every pull set containing $B_{k_2}$ is imperfect and contributes at least one charge to every frame containing $B_{k_2}$, including $F_z$, a contradiction.
	Hence, $|B_{k_2}| = 2|\varphi_2^{-1}(B_{k_2})| + 3$.
	Since we are not in Case \ref{case:Bk1Small} or Case \ref{case:Bk1Big},
		every frame with $\nu'$-charge zero must contain $B_{k_2}$ or $B_0$.
	
	Suppose there is a frame $F_{z'}$ containing $B_0$ and not containing $B_{k_2}$
		with $\nu'(F_{z'}) = 0$.
	Since we are not in Case \ref{case:B0allthreeblocks}, $F_{z'}$ contains at least one 2-block
		 and the proof of Case \ref{case:B0withTwoBlock} shows that $F_{z'}$ is the \emph{only} frame 
		 with $\nu'$-charge zero containing $B_0$ and not containing $B_{k_2}$.
	
	Therefore, there are at most $t + 1$ frames with $\nu'$-charge zero, whether or not there is a frame $F_{z'}$ with $\nu'(F_{z'}) = 0$ containing $B_0$ and not $B_{k_2}$
		and hence $\sum_{j : \nu'(F_j) > 0} [\nu'(F_j) - 1] \leq t.$
		
	If $|\varphi_2^{-1}(B_{k_2})| \geq 2$, then Claim \ref{claim:PullSetWithTwos} implies  $\sum_{j : \nu'(F_j)>0} [\nu'(F_j) - 1] \geq t+1$.
	If $|\varphi_2^{-1}(B_{k_2})| = 1$, then Claim \ref{claim:Bk2Big} implies  $\sum_{j : \nu'(F_j)>0} [\nu'(F_j) - 1] \geq t+1$.
	In either case we have a contradiction.
\end{mycases}	
	
	This completes the proof of Claim \ref{claim:bstar}
\end{proof}

	Thus, we have a block $B_*$ and a frame $F_z$ 
		so that $B_* \in F_z$,
		$\nu'(F_z) = 0$,
		and every 2-block $B_j$ has 
		$B_*, \psi_2(B_j), B_j,$ and $\varphi_2(B_j)$ 
		appearing in the cyclic order of blocks of $X$.
	Fix $B_j$ to be the first 2-block that appears after $B_*$ in the cyclic order.
	We will now prove that $a + b + c \geq 3$.
	
	Consider $\psi_2(B_j)$.
	Observe that $\varphi_2^{-1}(\psi_2(B_j)) = \emptyset$, 
		by the choice of $B_*$ and $B_j$.
	Hence, $a \geq |\psi_2(B_j)| - 4$.
	If $|\psi_2(B_j)| \geq 7$, then $a \geq 3$.
	Thus, $|\psi_2(B_j)| \in \{5,6\}$ and $\psi_2^{-1}(\psi_2(B_j)) = \{B_j\}$.
	
	Consider the frame $F_{j-t+1}$, whose last block is $B_j$.
	By the choice of $B_j$, all blocks in $F_{j-t+1} \setminus \{B_j\}$ 
		have size at least three, so $\sigma(F_{j-t+1}) \geq 3t-1$.
	This implies $\psi_2(B_j) \in F_{j-t+1}$.
	Since $\psi_2(B_j) \owns x_j - 3t$ and $|\psi_2(B_j)| \leq 6$, 
		there are at least $3t-4$ elements strictly between $\psi_2(B_j)$ and $B_j$
		which must be covered by at most $t-2$ blocks.
	Therefore, there exists some block $B_k$ strictly between $\psi_2(B_j)$ and $B_j$
		with $|B_k| \geq 4$.
	Select $B_k$ to be the first such block appearing after $\psi_2(B_j)$.

		
\begin{mycases}
	\case{$|\psi_2(B_j)| = 6$.} \label{case:BJEQ6}
	This implies $a \geq 2$.
	If $|B_k| \geq 5$, by choice of $B_j$ we have $\varphi_2^{-1}(B_k) = \emptyset$ and 
		$a \geq 3$.
	Therefore, $|B_k| = 4$ and $\psi_4(B_k)$ is a block of order at least four.
	If $|\psi_4(B_k)| \geq 5$, then $\varphi_2^{-1}(\psi_4(B_k)) = \emptyset$ and
		$a \geq 3$.
	Otherwise, $|\psi_4(B_k)| = 4$, and the frame $F_i$ starting at $B_i = \psi_4(B_k)$
		also contains $\psi_2(B_j)$ and $B_k$.
	Thus, $c = 1$ and $a + c \geq 3$.			
	
\case{$|\psi_2(B_j)| = 5$ and  $|B_k| \geq 5$.}\label{case:AK}
	Note that $\varphi_2^{-1}(B_k) = \emptyset$ by choice of $B_j$,
		which implies that $a \geq 2$.
	If $|B_k|\geq 6$, then $a \geq 3$; hence $|B_k| = 5$.
	Let $B_i = \psi_2(B_j)$ and consider the set
		$N_k = \{ x_k - 3t, x_k -3t+ 1, x_k - 3t + 5, x_k - 3t + 6\}$.
	The 
		elements in $N_k$ are non-neighbors with $x_k$ or $x_{k+1}$. 
	Since $X$ is a clique, $X$ is disjoint from 
		$N_k$.
	We must consider which elements in 
		$A_k = \{x_k-3t+2,x_k-3t+3,x_k-3t+4\}$
		are contained in $X$.
	If $B_*$ appears before $A_k$, then since $B_j$ is the first 2-block after $B_*$,
		there is at most one element of $X$ in $A_k$.
	If $B_*$ appears after $A_k$ and two elements of $A_k$ are in $X$, then they form a 2-block $B_{j'}$
		with $\varphi_2(B_{j'}) = B_k$, contradicting the choice of $B_*$.
	Hence, $|X \cap A_k| \leq 1$ and the elements from $X$ in 
		$A_k$ form either blocks of size at least five
		or two consecutive blocks of order at least four.
		
\begin{casefig}
			\scalebox{\casefigratio}{\begin{lpic}[]{"figs-unique/Case2"(,22.5mm)}
			   \lbl[b]{13.5,16.5;$A_k$}
			   \lbl[t]{56,19.5;$B_i = \psi_2(B_j)$}
			   \lbl[t]{82,19; $B_k$}
			   \lbl[t]{106,19;$B_j$}
			\end{lpic}}
	\caption{\label{fig:case:AK}Claim \ref{claim:itIScomplicated}, Case \ref{case:AK}.}
\end{casefig}

\begin{subcases}
	\subcase{$A_k \cap X = \emptyset$.}\label{subcase:AkXempty}
		Let $B_\ell$ be the block containing $x_k - 3t$.
		Note that $|B_\ell| \geq 8$.
		If $\varphi_2^{-1}(B_\ell) = \emptyset$, then $a \geq 4$.
		Otherwise $\varphi_2^{-1}(B_\ell) \neq \emptyset$, and  $B_*$ appears between $B_\ell$ and $B_i$.
		Then, there are at most $3t-7$ elements between $B_\ell$ and $B_k$.
		Since $|B_*| \geq 1$, $|B_i| \geq 5$, and 
			all other blocks have size at least three, 
			the $t-2$ blocks after $B_\ell$ 
			cover at least $3t-6$ elements.
		Thus, every frame containing $B_*$ (including $F_z$) must also contain $B_\ell$ or $B_k$.
		This implies that $\nu'(F_z) \neq 0$, a contradiction.

\begin{casefig}
			\scalebox{\casefigratio}{\begin{lpic}[]{"figs-unique/Case2.ii"(,22.5mm)}
			   \lbl[t]{10,20.5;$B_{\ell_1}$}
			   \lbl[t]{5,4;$F_{\ell_1}$}
			   \lbl[t]{26,20.5;$B_{\ell_2}$}
			   \lbl[t]{62,21.5;$B_i = \psi_2(B_j)$}
			   \lbl[t]{88,19; $B_k$}
			   \lbl[t]{112,19;$B_j$}
			\end{lpic}}
	\caption{\label{fig:subcase:AkX3t3}Claim \ref{claim:itIScomplicated}, Case \ref{subcase:AkX3t3}.}
\end{casefig}
	\subcase{$A_k \cap X = \{ x_k - 3t + 3\}$.}\label{subcase:AkX3t3}
		Then, the block starting at $x_k - 3t+3$ and the block preceding it have size at least four.
		These two blocks (call them $B_{\ell_1}$ and $B_{\ell_2}$) and $\psi_2(B_j)$ are contained in a single frame, $F_{\ell_1}$, so $c = 1$ and $a + c \geq 3$.

	\subcase{$A_k \cap X \neq \{ x_k - 3t  + 3\}$ and $B_*$ appears before $A_k$.}
		\label{subcase:AkXNOT3t3Before}
		Thus, the element in $A_k \cap X$ is either the first element in a block of size at least five
			or is the first element following a block of size at least five.
		In either case, this block, $B_\ell$, has $\varphi_2^{-1}(B_\ell) = \emptyset$, by the choice of $B_*$ and $B_j$.
		This implies $a \geq 3$.

	\subcase{$A_k \cap X \neq \{ x_k - 3t  + 3\}$ and $B_*$ appears between $A_k$ and $B_i$.}
		\label{subcase:AkXNOT3t3Between}
		Let $B_\ell$ be the block of size at least five that is guaranteed 
			by the element in $A_k \cap X$.
		There are at most $3t - 3$ elements between $B_\ell$ and $B_k$.
		Since $|B_*| \geq 1$, $|B_i| = 5$, and all other blocks between $B_\ell$ and $B_k$
			have size at least three,
			the $t-1$ blocks following $B_\ell$ cover at least $3t - 3$ elements.
		Thus, any frame containing $B_*$ also contains either 
			$B_\ell$ or $B_k$, and thus has positive charge.
		This includes $F_z$, but $\nu'(F_z) = 0$, a contradiction.
		
\end{subcases}

	\case{$|\psi_2(B_j)| = 5$ and all blocks between 
		$\psi_2(B_j)$ and $B_j$ have size at most four.} 
	Since there are $3t - 4$ elements strictly between $\psi_2(B_j)$ and $B_j$
		that must be covered by at most $t-2$ blocks of size at least three,
		there are at least two 4-blocks $B_k, B_{k'}$ between $\psi_2(B_j)$ and $B_j$.
	Thus, the blocks $B_{\ell_0} = \psi_2(B_j), B_{\ell_1} = B_k,$ and $B_{\ell_2} = B_{k'}$	
		are contained in a single frame and $c = 1$ giving $a + c \geq 3$.
	\end{mycases}
	

This completes the proof of Claim \ref{claim:itIScomplicated}.
\end{proof}

	Claims \ref{claim:complicateddischarging} and \ref{claim:itIScomplicated} imply that 
		an $r$-clique $X$ in $G + \{0,1\}$ has no 2-blocks.
	By Claim \ref{claim:threegenunique}, $G+\{0,1\}$ has a unique $r$-clique
		and hence $G$ is $r$-primitive.
\end{proof}

\section{Constructions of Sporadic Graphs}
\label{sec:constructions}

\def\constructionheight{1.75in}
\def\twoconstructionheight{2.5in}
\def\threeconstructionheight{1.75in}

In this section, we give explicit constructions for all known $r$-primitive graphs, 
	including those found in previous work.
It is a simple computation to verify that every graph presented is uniquely $K_r$-saturated, so proofs are omitted.
In addition to the descriptions given here, all graphs are available online\footnote{Graphs available in graph6 format or as adjacency matrices at \url{http://www.math.unl.edu/~shartke2/math/data/data.php}.}.

\subsection{Uniquely $K_4$-Saturated Graphs}

\begin{construction}[Cooper ~\cite{CooperWenger}, Figure \ref{fig:g10a}]\label{const:G10}
	$G_{10}$ is the graph built from two 5-cycles 
		$a_0$, $a_1$, $a_2$, $a_3$, $a_4$ and $b_0$, $b_1$, $b_2$, $b_3$, $b_4$ 
		where $a_i$ is adjacent to $b_{2i-1}, b_{2i},$ and $b_{2i+1}$.
\end{construction}

\begin{construction}[Collins~\cite{CooperWenger}, Figure \ref{fig:g12a}]\label{const:G12}
	The graph $G_{12}$ is the vertex graph of the icosahedron with a perfect matching added between antipodal vertices.
	Another description takes vertices $v_0, v_1$ and two 5-cycles $u_{j,0},\dots,u_{j,4}$ ($j \in \{0,1\}$)
		with $v_j$ adjacent to $v_{j+1}$ and $u_{j,i}$ for all $i \in [5]$ and
		$u_{0,i}$ adjacent to $u_{1, i}$, $u_{1,i+1}$, and $u_{1,i+3}$ for all $i \in \Z_5$.
\end{construction}

\begin{figure}[htp]
	\centering
	\mbox{
		\subfigure[\label{fig:g10a}Construction \ref{const:G10}, $G_{10}$.]{
		\begin{lpic}[]{"figs-unique/uk4-10b"(60mm,)}
		\lbl[]{35,125;$a_0$}
		\lbl[]{130,125;$a_1$}
		\lbl[]{230,125;$a_2$}
		\lbl[]{87,8;$a_3$}
		\lbl[]{180,8;$a_4$}
		\lbl[]{36,50;$b_0$}
		\lbl[]{83,85;$b_1$}
		\lbl[]{130,50;$b_2$}
		\lbl[]{178,85;$b_3$}
		\lbl[]{225,50;$b_4$}
		\end{lpic}
		}	
		\qquad
		\subfigure[\label{fig:g12a}Construction \ref{const:G12}, $G_{12}$.]{\includegraphics[width=40mm]{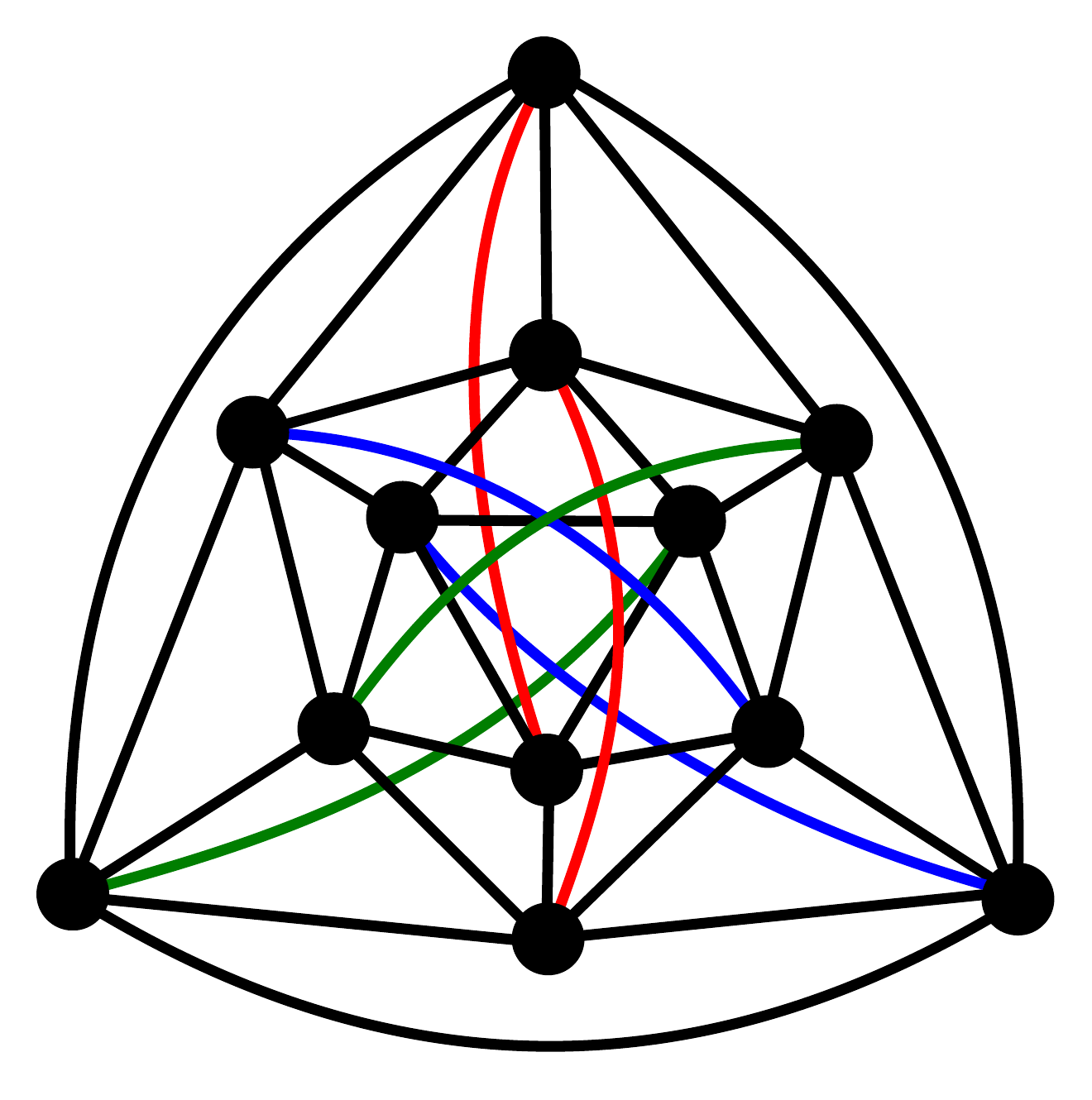}}
	}
	\mbox{
		\subfigure[\label{fig:g13a}Construction \ref{const:G13A}, $G_{13}$.]{\includegraphics[width=40mm]{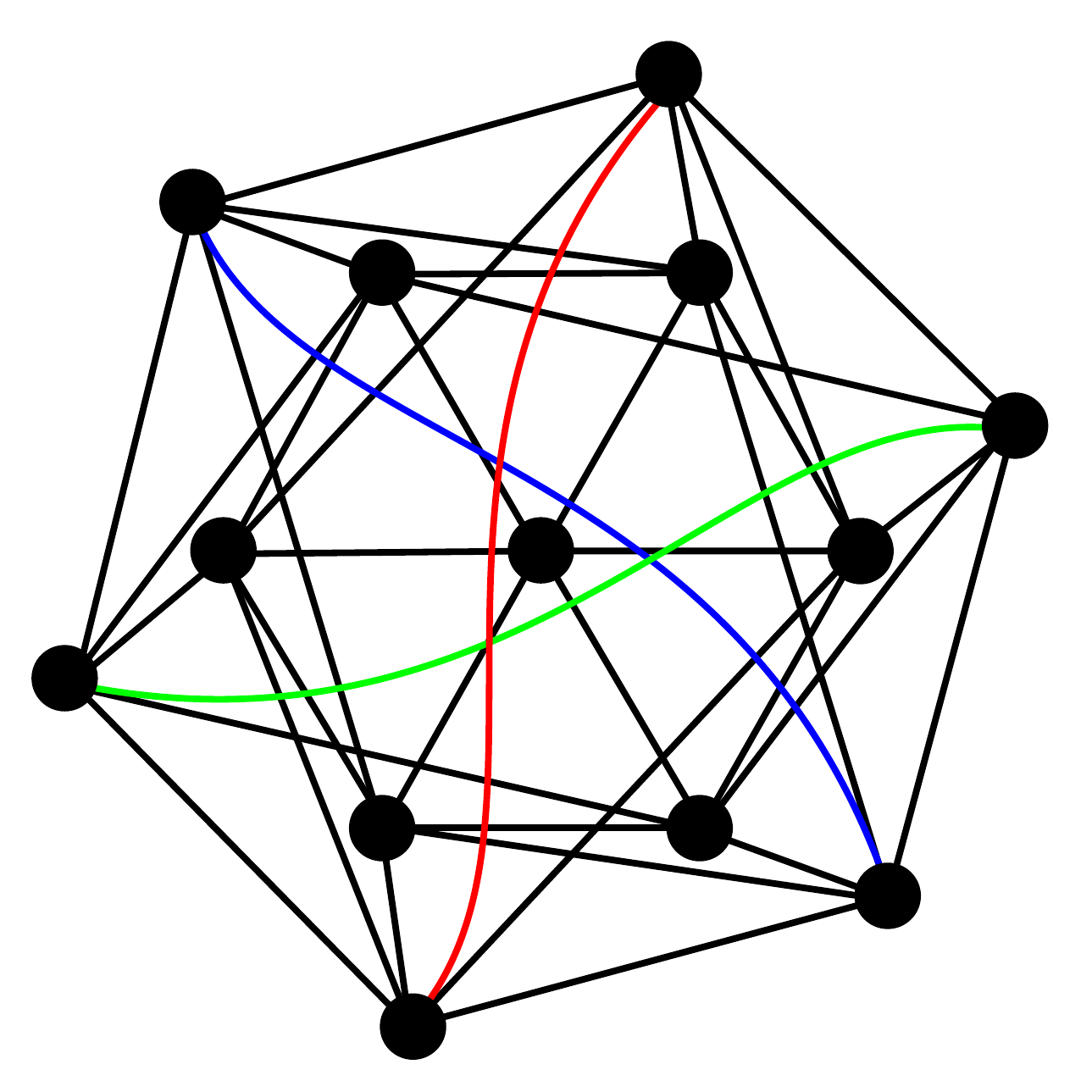}}
		\qquad
		\subfigure[\label{fig:g13b}Construction \ref{const:G13B}, Paley$(13)$.]{\hspace{1cm}\includegraphics[width=40mm]{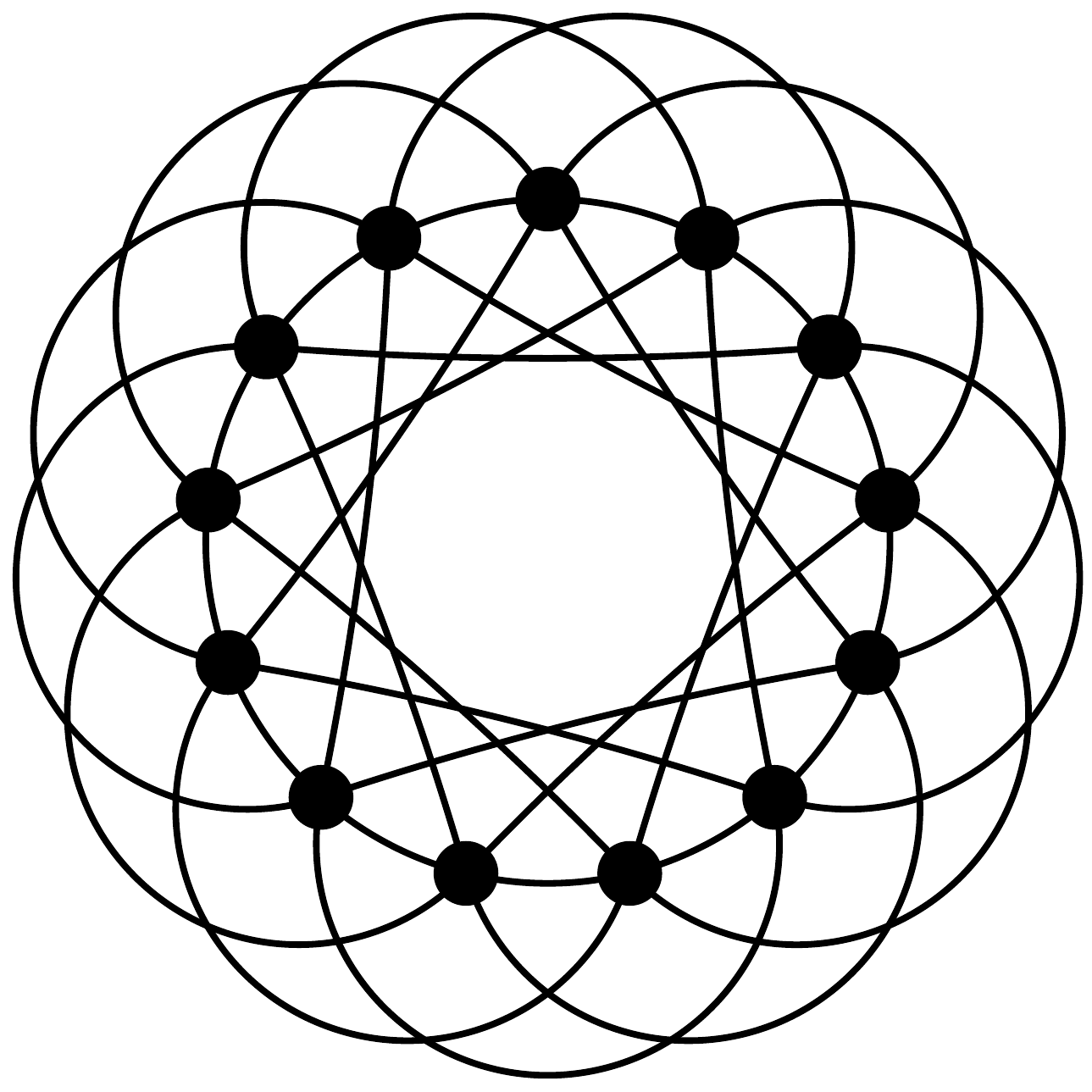}\hspace{1cm}}
	}
	\caption{\label{fig:uk4sat}Uniquely $K_4$-saturated graphs on $10$--$13$ vertices.}
\end{figure}

\begin{construction}[Figure \ref{fig:g13a}]\label{const:G13A}
	$G_{13}$ is given by vertices $x, y_1,\dots,y_6, z_1,\dots, z_6$, 
		where $x$ is adjacent to every $y_i$,
		$y_i$ and $y_{i+1}$ are adjacent for all $i \in \{1,\dots, 6\}$,
		and $z_i$ and $z_{i+1}$ are adjacent for all $i \in \{1,\dots, 6\}$.
	Further, $z_i$ is adjacent to $z_{i+3}$, $y_i$, $y_{i-1}$, and $y_{i+2}$.
\end{construction}

\begin{construction}[Figure \ref{fig:g13b}]\label{const:G13B}
	The Paley graph \cite{Paley} of order 13, Paley$(13)$, is isomorphic to the 
		Cayley complement $\cc{\Z_{13}}{\{1, 3, 4 \}}$.
\end{construction}

\begin{construction}[Figure \ref{fig:g18a}]\label{const:G18A}
Let $H$ be the graph on vertices $x, v_1,\dots,v_5$ with $x$ adjacent to every $v_i$ 
	and the vertices $v_1,\dots, v_5$ form a 5-cycle.
Note that $H$ is uniquely $K_4$-saturated, as $v_1,\dots,v_5$ induce $C_5$, which is 3-primitive.
$G_{18}^{(A)}$ has vertex set $V = \{1,2,3\}\times\{x,v_1,v_2,v_3,v_4,v_5\}$.
A vertex $(a,x)$ or $(a, v_i)$ in $V$ considers the number $a$ modulo three and $i$ modulo $5$.
The vertices $(a,x)$ with $a \in \{1,2,3\}$ form a triangle.
For each $a$, $(a,x)$ is adjacent to $(a,v_i)$ for each $i$ but is not adjacent to $(a+1,v_i)$ or $(a+2,v_i)$ for any $i$.
For each $a$ and $i$, the vertex $(a,v_i)$ is adjacent to $(a,v_{i-1})$ and $(a,v_{i+1})$ (within the copy of $H$)
	and also $(a+1,v_{i+2}), (a+1,v_{i-2}), (a-1,v_{i+2}), (a-1,v_{i-2})$ (outside the copy of $H$).
\end{construction}

\begin{construction}[Figure \ref{fig:g18b}]\label{const:G18B}
	Let $G_{18}^{(B)}$ have vertex set $\Z_2 \times \Z_9$ where each coordinate is taken modulo two and nine, respectively.
	For fixed $a$, the vertices $(a,i)$ and $(a,j)$ are adjacent if and only if $|i-j| \leq 2$.
	For fixed $i$, the vertex $(0,i)$ is adjacent to $(1, 2i), (1,2i+4)$ and $(1,2i+5)$.
	Conversely, for fixed $j$ the vertex $(1,j)$ is adjacent to $(0,5j), (0,5j+7)$ and $(0,5j+2)$.
\end{construction}

\begin{figure}[p]
	\centering
	\scalebox{0.66}{
	\begin{lpic}[]{"figs-unique/G18Alpic"(150mm,)}
		\lbl[b]{37.5,92.5;$(1,x)$}
		\lbl[b]{107,77;$(2,x)$}
		\lbl[b]{174,92.5;$(3,x)$}
		\lbl[t]{37.5,41;$(1,v_1)$}
		\lbl[t]{106.75,41;$(2,v_1)$}
		\lbl[t]{174,41;$(3,v_1)$}
		\lbl[br]{17,35;$(1,v_5)$}
		\lbl[br]{87,35;$(2,v_5)$}
		\lbl[br]{154,35;$(3,v_5)$}
		\lbl[tr]{22,9;$(1,v_4)$}
		\lbl[tr]{93,9;$(2,v_4)$}
		\lbl[tr]{159,9;$(3,v_4)$}
		\lbl[tl]{55,9;$(1,v_3)$}
		\lbl[tl]{121,9;$(2,v_3)$}
		\lbl[tl]{190,9;$(3,v_3)$}
		\lbl[tl]{58,32;$(1,v_2)$}
		\lbl[tl]{127,32;$(2,v_2)$}
		\lbl[tl]{195,32;$(3,v_2)$}
	\end{lpic}
	}
	
	\includegraphics[height=1em]{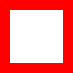} --- $\{ (2,v_1)\} \cup \left(N((2,v_1)) \cap \{ (j,v_i) : j \in \{1,3\}, i \in \{1,\dots,5\}\}\right)$.
	
	\caption{\label{fig:g18a}Construction \ref{const:G18A}, $G_{18}^{(A)}$, is 4-primitive, 7-regular, on 18 vertices.}
\end{figure}

\begin{figure}[p]
	\centering
	\scalebox{0.66}{
	\begin{lpic}[]{"figs-unique/G18Blpic"(150mm,)}
		\lbl[tr]{23,43; $(0,6)$}
		\lbl[r]{20,75; $(0,7)$}
		\lbl[br]{34,100; $(0,8)$}
		\lbl[tr]{45,15; $(0,5)$}
		\lbl[tl]{90,15; $(0,4)$}
		\lbl[tl]{112,43; $(0,3)$}
		\lbl[l]{114,76; $(0,2)$}
		\lbl[bl]{95,100; $(0,1)$}
		\lbl[b]{67,112; $(0,0)$}
		
		\lbl[tr]{159,43; $(1,6)$}
		\lbl[r]{158,76; $(1,7)$}
		\lbl[br]{170,100; $(1,8)$}
		\lbl[tr]{181,15; $(1,5)$}
		\lbl[tl]{226,15; $(1,4)$}
		\lbl[tl]{248,43; $(1,3)$}
		\lbl[l]{252,75; $(1,2)$}
		\lbl[bl]{231,100; $(1,1)$}
		\lbl[b]{203,112; $(1,0)$}
	\end{lpic}
	}
	
	\includegraphics[height=1em]{figs-unique/RedSquare} --- $\{ (0,1)\} \cup \left(N((0,1)) \cap \{ (1,i) : i \in \Z_9\}\right)$.
	
	\includegraphics[height=1em]{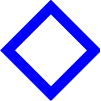} --- $\{ (1,0)\} \cup \left(N((1,0)) \cap \{ (0,i) :  i \in \Z_9\}\right)$.
	
	\caption{\label{fig:g18b}Construction \ref{const:G18B}, $G_{18}^{(B)}$, is 4-primitive, 7-regular, on 18 vertices.}
\end{figure}


\subsection{Uniquely $K_5$-Saturated Graphs}

\begin{construction}[Figure \ref{fig:g16a}]\label{const:G16A}
	Let $G_{16}^{(A)}$ have vertex set $\{v_1, v_2\} \bigcup \left(\{1,2\} \times \Z_7\right)$.
	The vertices $v_1$ and $v_2$ are adjacent.
	For each $j \in \{1,2\}$ and $i \in \Z_7$, $v_j$ is adjacent to $(j,i)$ and $(j,i)$ is adjacent to $(j, i+1), (j,i+2), (j,i-1)$ and $(j,i-2)$.
	(Hence, the subgraph induced by $(j,i)$ for fixed $j$ and $i \in \Z_7$ is isomorphic to $C_7^2$.)
	For $i \in \Z_7$, the vertex $(1,i)$ is adjacent to $(2, 2i), (2, 2i+1), (2, 2i-1)$, and $(2,2i-3)$.
	Conversely, for $i \in \Z_7$, the vertex $(2, i)$ is adjacent to $(1, 4i), (1, 4i-2), (1, 4i+3)$, and $(1, 4i-3)$.
\end{construction}

An interesting feature of $G_{16}^{(A)}$ is that it is not regular: $v_1$, and $v_2$ have degree 8 while the other vertices have degree 9.
This is a counterexample to previous thoughts that all uniquely $K_r$-saturated graphs with no dominating vertex were regular.

\begin{construction}[Figure \ref{fig:g16b}]\label{const:G16B}
	The graph $G_{16}^{(B)}$ has vertex set 
			$\{ x\} %
			\cup \{u_i : i \in \Z_3 \} %
			\cup \{v_j : j \in \Z_6 \} %
			\cup \{ z_{k,i} : k \in \{0,1\}, i \in \Z_3 \}$.
	The vertex $x$ is adjacent to $u_i$ for all $i \in \Z_3$ and $v_j$ for all $j \in \Z_6$.
	There are no edges among the vertices $u_i$.
	The vertices $v_j$ form a cycle, with an edge $v_jv_{j+1}$ for all $j \in \Z_6$.
	The vertices $z_{k,i}$ form a complete bipartite graph,
		with an edge $z_{0,i}z_{1,j}$  for all $i, j \in \Z_3$.
	For $i \in \{0,1,2\}$, 
		the vertex $u_{i}$ is adjacent to 
			$v_{2i - 1}$, $v_{2i}$, $v_{2i+1}$, and $v_{2i+2}$,
		and adjacent to $z_{k, i + 1}$ and $z_{k, i - 1}$ for $k \in \{0, 1\}$.
	For $i \in \{0,1,2\}$, the vertex $z_{0,j}$ is adjacent to 
		$v_{2i}$, $v_{2i+1}$, $v_{2i+2}$, and $v_{2i+4}$,
		while the vertex $z_{1,i}$ is adjacent to $v_{2i-1}$, $v_{2i}$, $v_{2i+1}$, and $v_{2i+3}$.
\end{construction}

\begin{figure}[p]
	\centering
	\scalebox{0.66}{
	\begin{lpic}[]{"figs-unique/G16Alpic"(150mm,)}
		\lbl[br]{31,90;$(1,6)$}
		\lbl[tr]{18,45;$(1,5)$}
		\lbl[tr]{45,15;$(1,4)$}
		\lbl[tl]{92,15;$(1,3)$}
		\lbl[tl]{116,45;$(1,2)$}
		\lbl[bl]{108,90;$(1,1)$}
		\lbl[b]{70,110;$(1,0)$}
		\lbl[t]{70,67;$v_1$}
		
		\lbl[br]{163,90;$(2,6)$}
		\lbl[tr]{155,45;$(2,5)$}
		\lbl[tr]{178,15;$(2,4)$}
		\lbl[tl]{226,15;$(2,3)$}
		\lbl[tl]{251,45;$(2,2)$}
		\lbl[bl]{242,90;$(2,1)$}
		\lbl[b]{203,110;$(2,0)$}
		\lbl[t]{202,67;$v_2$}
	\end{lpic}
	}
	
	\includegraphics[height=1em]{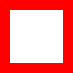} --- $\{ (1,1)\} \cup \left(N((1,1)) \cap \{ (2,i) :  i \in \{0,1,\dots,6\}\}\right)$.
	
	\includegraphics[height=1em]{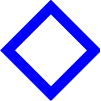} --- $\{ (2,0)\} \cup \left(N((2,0)) \cap \{ (1,i) : i \in \{0,1,\dots,6\}\}\right)$.
	\caption{\label{fig:g16a}Construction \ref{const:G16A}, $G_{16}^{(A)}$, is 5-primitive \emph{and irregular}, on 16 vertices.}
\end{figure}

\begin{figure}[p]
	\centering
	\scalebox{0.66}{
	\begin{lpic}[]{"figs-unique/G16Blpic"(150mm,)}
		\lbl[]{69.75,61.5;$x$}
		\lbl[b]{70,93;$u_0$}
		\lbl[l]{100,44;$u_1$}
		\lbl[r]{40,44;$u_2$}
		\lbl[tl]{100,115;$v_0$}
		\lbl[l]{126,62;$v_1$}
		\lbl[tl]{100,15;$v_2$}
		\lbl[tr]{40,15;$v_3$}
		\lbl[r]{12,62;$v_4$}
		\lbl[tr]{40,115;$v_5$}
		\lbl[]{170,100;$z_{0,0}$}
		\lbl[]{170,62;$z_{0,1}$}
		\lbl[]{170,25;$z_{0,2}$}
		\lbl[]{255,100;$z_{1,0}$}
		\lbl[]{255,62;$z_{1,1}$}
		\lbl[]{255,25;$z_{1,2}$}
	\end{lpic}
	}
	
	\includegraphics[height=1em]{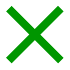} --- $\{ u_0 \} \cup \left(N(u_0) \cap \{ z_{j,i} : j\in \{0,1\},  i \in \Z_3\}\right)$.
	
	\includegraphics[height=1em]{figs-unique/RedSquare} --- $\{ z_{0,0} \} \cup \left(N(z_{0,0}) \cap \{ v_i : i \in \Z_6\}\right)$.
	
	\includegraphics[height=1em]{figs-unique/BlueDiamond} --- $\{ z_{1,0} \} \cup \left(N(z_{1,0}) \cap \{ v_i :  i \in \Z_6\}\right)$.
	\caption{\label{fig:g16b}Construction \ref{const:G16B}, $G_{16}^{(B)}$, is 5-primitive, 9 regular, on 16 vertices.}
\end{figure}

\subsection{Uniquely $K_6$-Saturated Graphs}

\begin{figure}[p]
	\centering
	\scalebox{0.6}{
	\begin{lpic}[]{"figs-unique/G15Alpic"(100mm,)}
		\lbl[]{10,84;$x$}
		
		\lbl[]{12,33;$v_1$}
		\lbl[]{12,68;$v_0$}
		
		\lbl[]{64,80;$u_0$}
		\lbl[]{64,62;$u_1$}
		\lbl[]{64,40;$u_2$}
		\lbl[]{64,21;$u_3$}
		
		\lbl[]{79,80;$c_0$}
		\lbl[]{119,80;$c_1$}
		\lbl[]{119,56;$c_2$}
		\lbl[]{79,56;$c_3$}
		
		\lbl[]{78,40;$q_0$}
		\lbl[]{78,13;$q_3$}
		\lbl[]{119,13;$q_2$}
		\lbl[]{119,40;$q_1$}
	\end{lpic}
	}
	
	\includegraphics[height=1em]{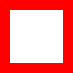} --- $\{ u_3 \} \cup \left(N(u_3) \cap \{ q_i :  i \in [4]\}\right)$.
	
	\includegraphics[height=1em]{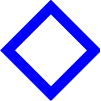} --- $\{ c_1 \} \cup \left(N(c_1) \cap \{ q_i : i \in [4]\}\right)$.
	
	\includegraphics[height=1em]{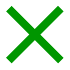} --- $\{ v_1 \} \cup \left(N(v_1) \cap \{ c_i, q_i : i \in [4]\}\right)$.
	
	\caption{\label{fig:g15a}Construction \ref{const:G15A}, $G_{15}^{(A)}$, is 6-primitive, 10 regular, on 15 vertices.}
\end{figure}

\begin{construction}[Figure \ref{fig:g15a}]\label{const:G15A}
	The graph $G_{15}^{(A)}$ has vertices $x$, $v_0$, $v_1$, $u_1$, $\dots$, $u_4$, $c_1$, $\dots$, $c_4$, $q_1$, $\dots$, $q_4$.
	The vertex $x$ dominates all but the $q_i$'s.
	The vertices $v_0, v_1$ are adjacent and dominate the $u_i$'s.
	Also, $v_i$ dominates $c_{2i}, c_{2i+1}, q_{2i}, q_{2i+1}$ for each $i \in \Z_2$.
	The vertices $u_0$ and $u_2$ are adjacent as well as $u_1$ and $u_3$.
	The vertices $u_i$ dominate the vertices $c_j$.
	Also, the vertex $u_i$ is adjacent to $q_j$ if and only if $i \neq j$.
	The vertices $c_1,\dots,c_4$ form a cycle with edges $c_ic_{i+1}$.
	The vertices $q_1,\dots, q_4$ form a clique.
	The vertices $c_i$ and $q_j$ are adjacent if and only if $i \neq j$.
\end{construction}

\begin{figure}[p]
	\centering
	\scalebox{0.6}{
	\begin{lpic}[]{"figs-unique/G15Blpic"(150mm,)}
		\lbl[]{61,103;$q_0$}
		\lbl[]{103,72;$q_1$}
		\lbl[]{90,20;$q_2$}
		\lbl[]{33,20;$q_3$}
		\lbl[]{17,72;$q_4$}
		
		\lbl[]{161,114;$c_{1,0}$}
		\lbl[]{190,95;$c_{1,1}$}
		\lbl[]{182,72;$c_{1,2}$}
		\lbl[]{143,72;$c_{1,3}$}
		\lbl[]{133,95;$c_{1,4}$}
		
		\lbl[]{161,39;$c_{2,0}$}
		\lbl[]{189,30;$c_{2,1}$}
		\lbl[]{182,10;$c_{2,2}$}
		\lbl[]{143,10;$c_{2,3}$}
		\lbl[]{133,30;$c_{2,4}$}
	\end{lpic}
	}
	
	\includegraphics[height=1em]{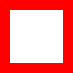} --- $\{ q_0 \} \cup \left(N( q_0 ) \cap \{ c_{j,i} : j \in \{1,2\}, i \in \Z_5\}\right)$.
	
	\caption{\label{fig:g15b}Construction \ref{const:G15B}, $G_{15}^{(B)}$, is 6-primitive, 10 regular, on 15 vertices.}
\end{figure}

\begin{construction}[Figure \ref{fig:g15b}]\label{const:G15B}
	The graph $G_{15}^{(B)}$ has vertices $q_i, c_{1,i}$, and $c_{2,i}$ for each $i \in \Z_5$.
	The subgraph induced by vertices $q_i$ is a 5-clique.
	For each $j \in \{1,2\}$, the subgraph induced by vertices $c_{j,i}$ for $i \in \Z_5$ is isomorphic to $C_5$
		with edges $c_{j,i}c_{j,i+1}$ between consecutive elements.
	For each $i, i' \in \Z_5$, there is an edge between $c_{1,i}$ and $c_{2,i'}$.
	For each $i \in \Z_5$, the vertex $q_i$ is adjacent to $c_{1,i}, c_{1,i-1},$ and $c_{1,i+1}$ 
		as well as $c_{2,2i}$, $c_{2,2i-1}$, and $c_{2,2i+2}$.
\end{construction}

\begin{figure}[p]
	\centering
	\scalebox{0.6}{
	\begin{lpic}[]{"figs-unique/G16Clpic"(150mm,)}
		\lbl[]{67,105;$c_{0}$}
		\lbl[]{105,90;$c_{1}$}
		\lbl[]{118,58;$c_{2}$}
		\lbl[]{105,25;$c_{3}$}
		\lbl[]{67,10;$c_{4}$}
		\lbl[]{28,25;$c_{5}$}
		\lbl[]{17,58;$c_{6}$}
		\lbl[]{28,90;$c_{7}$}
		
		\lbl[]{140,108;$q_{1,0}$}
		\lbl[]{188,108;$q_{1,1}$}
		\lbl[]{188,77;$q_{1,2}$}
		\lbl[]{140,77;$q_{1,3}$}
		
		\lbl[]{140,45;$q_{2,0}$}
		\lbl[]{188,45;$q_{2,1}$}
		\lbl[]{188,12;$q_{2,2}$}
		\lbl[]{140,12;$q_{2,3}$}
	\end{lpic}
	}
	
	 \includegraphics[height=1em]{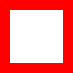} --- $\{ q_{1,0} \} \cup \left(N(q_{1,0}) \cap \{  c_i: i \in \Z_8 \} \right)$.
	
	\includegraphics[height=1em]{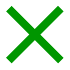} --- $\{ q_{1,1}\} \cup \left(N(q_{1,1}) \cap \{ q_{2,i} : i \in \Z_4 \}\right)$.
	
	\includegraphics[height=1em]{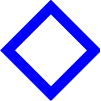} --- $\{ q_{2,1} \} \cup \left(N(q_{2,1}) \cap \{ c_i : i \in \Z_8 \}\right)$.
	
	\caption{\label{fig:g16c}Construction \ref{const:G16C}, $G_{16}^{(C)}$, is 6-primitive, 10 regular, on 16 vertices.}
\end{figure}

\begin{construction}[Figure \ref{fig:g16c}]\label{const:G16C}
	The graph $G_{16}^{(C)}$ is composed of three disjoint induced subgraphs isomorphic to $K_4, K_4$, and $\overline{C_8}$.
	Let the vertices $q_{0,0}, \dots, q_{0,3}$, and $q_{1,0}, \dots, q_{1,3}$ be the two copies of $K_4$ and vertices $c_0, \dots, c_7$ be the $\overline{C_8}$, 
		where the non-edges are for consecutive elements $(0,i)$ and $(0,i+1)$.
	For $i \in \{0,1,2,3\}$, the vertex $q_{1,i}$ is adjacent to $c_{2i + d}$ for all $d \in \{0,1,2,3,4,5\}$.
	For $i \in \{0,1,2,3\}$, the vertex $q_{2,i}$ is adjacent to $c_{2i + d}$ for all $d \in \{0,1,3,4,5,6\}$.
	For $i \in \Z_4$, the vertex $q_{1,i}$ is adjacent to $q_{2,i + 1}$ and $q_{2,i - 1}$.
\end{construction}

\section*{Acknowledgements}

We thank  David Collins, Joshua Cooper, Bill Kay, and Paul Wenger for sharing their early observations on this problem.
We also thank Jamie Radcliffe for contributing to the averaging argument found in Claim \ref{claim:twogenclique}.

\end{document}